%% file: FagGodJar_2014.tex
\documentclass[a4paper,10pt]{article}

\usepackage[utf8]{inputenc}
\usepackage[T1]{fontenc}
\usepackage[english]{babel}

\usepackage{amsmath,amsfonts,amssymb,amsthm,amscd}
\usepackage{graphicx,tabularx,color}
\usepackage[all]{xy}
 
\usepackage[hang]{caption}
\newtheorem{thm}{Theorem}[section]
\newtheorem{lem}[thm]{Lemma}
\newtheorem{coro}[thm]{Corollary}
\newtheorem{rem}[thm]{Remark}
\newtheorem{dfn}[thm]{Definition}
\newtheorem{thmalpha}{Theorem}

\newcommand{\Z}{\mathbb{Z}}
\newcommand{\N}{\mathbb{N}}
\newcommand{\Nstar}{\mathbb{N}^{\star}}
\newcommand{\R}{\mathbb{R}}
\newcommand{\C}{\mathbb{C}}
\newcommand{\CC}{\widehat{\mathbb{C}}}
\newcommand{\Hr}{\mathbb{H}_{\mathrm{r}}}
\newcommand{\D}{\mathbb{D}}

\DeclareMathOperator{\diam}{diam}
\DeclareMathOperator{\dist}{dist}
\DeclareMathOperator{\area}{area}

\newcommand{\Julia}{\mathcal{J}}
\newcommand{\Fatou}{\mathcal{F}}
\newcommand{\classB}{\mathcal{B}}
\newcommand{\classS}{\mathcal{S}}
\newcommand{\neighbor}{\mathcal{N}}

\newcommand{\scs}{\scriptsize}

\title{Wandering domains for composition of entire functions}
\date{\today}

\author{N{\'u}ria Fagella, S\'ebastien Godillon, and Xavier Jarque}

\begin{document}


\maketitle

\begin{abstract}
	 C.~Bishop in \cite[Theorem 17.1]{QuasiconformalFoldings} constructs an example of an entire function $f$ in class $\classB$ with {\em at least} two grand orbits of oscillating wandering domains. In this paper we show that his example has {\em exactly} two such orbits, that is, $f$ has no {\em unexpected} wandering domains. We apply this result to the classical problem of relating the Julia sets of composite functions with the Julia set of its members. More precisely, we show the existence of two entire maps $f$ and $g$ in class $\classB$  such that the Fatou set of $f\circ g$ has a wandering domain, while all Fatou components of $f$ or $g$ are preperiodic. This complements a result of A.~Singh in \cite[Theorem 4] {CompositionEntireFunctions} and results of W.~Bergweiler and A.Hinkkanen in \cite{SemiconjugationEntireFunctions} related to this problem.
\end{abstract}

\tableofcontents

\newpage

\section{Introduction}

The systematic global study of the phase portrait of dynamical systems given by the iterates of holomorphic maps of the complex plane goes back to the work of Pierre Fatou and Gaston Julia, at the beginning of the twentieth century. They developed a theory based on the concept of normal families and described precisely what nowadays is known as the Fatou and Julia sets. 

If $f:\C\to\C$ is a transcendental entire function (similar definitions apply if $f$ is a rational function on the Riemann sphere, or if $f$ is a transcendental meromorphic function on the complex plane), the \textit{Julia set}, $\Julia(f)$, is defined as the set of points $z\in\C$ for which the family of iterates $\{f^n\}_{n\geqslant 0}$ fails to be normal in the sense of Montel in every neighborhood  of $z$; that is, if $V$ is a neighborhood of a point in the Julia set, the infinite union of iterates of $V$ must cover the whole plane with the exception of, at most, two points.  By definition $\Julia(f)$ is closed and it can be proven that it is infinite \cite{IterationEntireFunctions}. Its complementary domains in $\C$, if any, are called \textit{Fatou domains} or \textit{Fatou components} and the union of all those components is called the {\it Fatou set}, denoted by $\Fatou(f)$. Both, the Julia and Fatou sets, are  invariant under $f$ and under all branches of $f^{-1}$, hence they form a natural dynamical partition of the complex plane. For a general discussion about their properties for both rational and transcendental dynamics see for example \cite{DynamicalPropertiesEntireFunctions,IterationMeromorphicFunctions,MilnorBook,QuasiconformalSurgeryBook}.

It follows that Fatou components correspond to maximal domains on which the dynamics of 
$f$ are in some sense tame or non chaotic.  In the language of sequences, if $U$ is a Fatou component then the family of iterates $\{f^{n}|_U\}_{n\geqslant 0}$ has a  subsequence  $\{f^{n_{k}}|_U\}_{k\geqslant 0}$ which, in compact subsets of $U$, converges uniformly to either a holomorphic map $g$ or to infinity. Moreover, due to the rigidity of holomorphic maps, the possible \textit{limit functions} $g$ reduce to either constants or irrational rotations in $U$. More precisely, one can show that  there are potentially few possibilities. If $U$ is a  Fatou component then $U$ is periodic if $f^n(U)\subset U$ for some $n\in\N$, but otherwise $U$ is
\begin{itemize}
\item \textit{wandering} if $f^{n}(U)\cap f^{m}(U)=\emptyset$  for all integers $n\neq m$, or \item  \textit{preperiodic} if it is eventually mapped into a periodic component.
\end{itemize}

For a general entire transcendental map,  periodic components can only be basins of attraction of attracting or parabolic cycles, Siegel disks or  Baker domains (also known as parabolic domains at infinity), depending on the limit behavior of the convergent subsequences. We refer to any of the references above for precise definitions. In this paper we shall concentrate mainly on wandering domains, rather than (pre)periodic components. 

Most types of periodic Fatou components are somehow associated to the orbit of a singular value of $f$. We say that $v$ is a {\em singular value} of $f$ if not all branches of $f^{-1}$ are well defined in every sufficiently small neighborhood of $f$. Let $S(f)$ denote the set of singular values of $f$, which is closed. For rational maps the set $S(f)$ is finite and formed exclusively by {\em critical values}, that is by images of {\em critical points} which are zeros of $f'$. Transcendental maps may additionally have {\em asymptotic values} i.e. points which morally have some of their preimages at infinity (like $v=0$ for the exponential map), or limit points of critical and asymptotic values. Some special classes of transcendental maps are more likely to dynamically resemble polynomial or rational dynamics, namely the  \textit{Speiser class}, $\classS$, of maps for which $S(f)$ is finite (also known as critically finite maps); or the  \textit{Eremenko-Lyubich's class},  $\classB$, of maps for which $S(f)$ is bounded. 

It was shown by Sullivan \cite{NoWanderingTheorem} in 1985 that rational maps have no wandering domains. He proved this major result using quasiconformal analysis, showing  that if a rational map  had a wandering domain, the space of quasiconformal deformations would have infinite dimension, contradicting the fact that the space of rational maps of degree $d$ is finite dimensional. Right after Sullivan's non wandering Theorem, Golberg and Keen \cite{FinitenessTheoremEntireFunctions}, and Eremenko and Lyubich \cite{DynamicalPropertiesEntireFunctions}, independently, generalized Sullivan's argument to show that wandering domains are not allowed either for transcendental entire maps in class $\classS$. See also \cite{WanderingDomains}.

Despite these negative results, it is well known that wandering domains are possible for  transcendental entire functions in general. The first example of a wandering domain was given by Baker \cite{EntireFunctionWanderingDomain} and later Herman \cite{ExemplesRationnellesOrbiteDense} proposed a systematic way of constructing wandering domains by means of lifting transcendental maps in $\C\setminus\{0\}$. On the other hand, if $U$ is a multiply connected Fatou domain, then it should be a wandering domain \cite{WanderingDomains}. The wandering domains constructed in  \cite{MultiplyConnectedDomains,EntireFunctionWanderingDomain} happen to be multiply connected (nowadays known as \textit{Baker-wandering domain}). His examples have been the seed of a large series of papers on multiply connected wandering domains. See \cite{OnMultiplyConnectedWanderingDomains,MultiplyConnectedWanderingDomains} and references therein.

It was proven by Fatou \cite[Section 28]{Fatou1}, that if $U$ is a wandering domain of $f$, all limit functions of any convergent subsequence $\{f^{n_{k}}|_U\}_{k\geqslant 0}$ are constant, say $a\in\CC$. Later on, Baker \cite{LimitFunctions} proved that if $a\in\CC$ is a limit function then either $a=\infty$ or $a\in\Julia(f)\cap\overline{E}$ where
$$
E:=\bigcup_{s\in S(f)}\bigcup_{n\geqslant 0} f^{n}(s)
$$
is the \textit{post-singular set}. This result was improved later in \cite{LimitFunctionsIteratesWanderingDomains} by showing that either $a=\infty$ or $a\in\Julia(f)\cap E'$ where $E'$ denotes the \textit{derived set} of $E$ (i.e., finite limit points of $E$). As an application of this result, it is easy to deduce that $f(z)=\exp(z)$ has no wandering domains (and hence $\Julia(f)=\C$) since, on the one hand, $E'=\emptyset$ and, on the other hand, the expansivity of the map to the right implies that $a$ cannot be $\infty$. Consequently, wandering domains can be classified in terms of their constant limit functions as \textit{escaping} if $f^{n}|_U\to\infty$  as $n\to+\infty$ (locally uniformly in $U$), \textit{oscillating} if there exist strictly increasing sequences $(n_{k})_{k\geqslant 0}$ and $(m_{k})_{k\geqslant 0}$ such that $f^{n_{k}}|_U\to\infty$ and $f^{m_{k}}|_U\to a\in\C$ as $k\to+\infty$, and \textit{(orbitally) bounded} if $\infty$ is not the limit function of any convergent subsequence $\{f^{n_{k}}|_U\}_{k\geqslant 0}$ (i.e, all limit functions are finite). All previously mentioned examples where of the first kind and it was not until 1987 that Eremenko and Lyubich \cite{EntireFunctionsPathologicalDynamics} gave an example of a transcendental entire function with a wandering domain of the second kind with infinitely many different finite constant limit functions. This example was constructed using approximation theory and hence it is not explicit. As of today it is still an open question whether orbitally bounded wandering domains exist. 

Fatou components of functions in class $\classB$ are always simply connected, and therefore such maps have no Baker-wandering domain. Even more,  it was proven in \cite{DynamicalPropertiesEntireFunctions} that $\Julia(f)=\overline{\mathcal{I}(f)}$, where $\mathcal{I}(f)$ is the {\em escaping set}, i.e. the set of points whose orbit escapes to $\infty$. Hence in class $\classB$ there are no escaping wandering domains. Recently, however, two new results have opened new lights about the existence of wandering domains in class $\classB$. These results are relevant in the paper at hand.

In \cite{AbsenceWanderingDomains} Rempe-Gillen and Mihaljevi\'c-Brandt give a set of conditions for a function, under which wandering domains can be discarded.   Since in class $\classB$ any wandering domain should have some finite limit functions, they conjecture that if the orbit of all singular values converge uniformly to infinity, i.e. if
$$
\lim_{n\to+\infty}\ \inf_{s\in S(f)}\ |f^{n}(s)|=+\infty
$$
then $f$ has no wandering domains. They prove that this is indeed the case if $f$ satisfies a certain somewhat technical condition. One of the key tools they use to prove their remarkable result is  
\cite[Theorem 4.1]{AbsenceWanderingDomains} (see also Theorem \ref{thm:Rempe}) which we shall use in our proofs.

The second breakthrough appears in \cite{QuasiconformalFoldings}
where Bishop gives a constructive method to build transcendental entire functions in class $\classB$, with  certain control on their singular set and, potentially, on their dynamics. Roughly speaking, Bishop's main theorem is as follows (see  Section \ref{sec:preliminaries} for details). Let $T$ be an unbounded connected bipartite graph embedded in $\C$ and let $\{\tau_{j}\}_{j}$ be a collection of conformal maps from every connected component $\Omega_{j}$ of $\C\setminus T$ onto either the right half plane $\Hr$ or the unit disk $\D$. Let $\{\sigma_{j}\}_{j}$ be a collection of certain holomorphic maps (chosen conveniently) from $\tau_{j}(\Omega_{j})\in\{\Hr,\D\}$ into $\C$. Then there exists a transcendental entire function $f$ which coincides with $\sigma\circ\tau$ outside a certain neighborhood of $T$. Moreover: \textbf{(i)} $f$ is in class $\classB$ and all its singular set belongs to  $\D\cup\{-1,+1\}$; \textbf{(ii)} $f$ sends (a slight enlargement of) $T$ to $[-1,1]\cup\mathbb{S}^{1}$; and \textbf{(iii)} the vertices of  $T$ are the critical points of $f$ with critical values in $\pm 1$. See \cite{QuasiconformalFoldings} and Sections \ref{sec:preliminaries} and \ref{sec:Bishop_example} of this paper for details.

Bishop's construction method can be used to provide examples of many different phenomena. In particular, he adapts it to prove the following. 

\begin{thm}[{\cite[Theorem 17.1]{QuasiconformalFoldings}}]\label{thm:Bishop_example}
There exists a transcendental entire function $f$ in class $\classB$ such that $f$ has two grand orbits of oscillating wandering domains.
\end{thm}

We remark that Bishop's example is not a negative answer to the mentioned conjecture of Rempe-Gillen and Mihaljevi\'c-Brandt since infinitely many of the critical points are contained in the oscillating wandering domains. 


The first goal in this paper is to prove the following result concerning the map $f$ in Theorem \ref{thm:Bishop_example}.

\begin{thmalpha}\label{thm:A}
Let $f$ be given by Theorem \ref{thm:Bishop_example}. Then $f$ has  {\bf exactly} two grand orbits of wandering domains. 
\end{thmalpha}

We informally say that $f$ has no {\em unexpected } wandering domains. To prove Theorem A we shall use some ideas developed in \cite{AbsenceWanderingDomains} (see Section \ref{sec:proof_thmA}).

Our second goal is to show what we believe is an interesting application of the previous result to the dynamics of composite functions. Given two holomorphic maps $f$ and $g$ we consider the relationship between the Fatou and Julia sets of $f$ and $g$, and that of $g\circ f$ or $f\circ g$. One particular case is when $f$ and $g$ commute. i.e., $f\circ g=g\circ f$. It was proven by Baker \cite{WanderingDomains} that if $f$ and $g$ are commuting rational maps  then $\Julia(f) =\Julia(g)$, but it is still an open problem to prove the same property for commuting entire transcendental functions. The best partial result is due to Bergweiler and Hinkkanen \cite{SemiconjugationEntireFunctions} who proved that $\Julia(f)=\Julia(g)$ whenever $f$ and $g$ have no wandering domains (their result implies this statement).

For the general case (that is without assumming that $f$ and $g$ commute),  Bergweiler and Wang \cite{CompositeEntireFunctions} proved that for every $z\in\C$, $z \in \Julia(f\circ g)$ if and only if $g(z)\in \Julia(g\circ f)$. Moreover if $U$ is a component of $\Fatou(f\circ g)$ and $V$ is the component of $\Fatou(g\circ f)$ that contains $g(U)$, then $U$ is wandering for $f\circ g$ if and only if $V$ is wandering for $g\circ f$. A natural question to ask is whether the composition of two maps may have wandering domains if none of the members do. Related to this question, Singh 
\cite[Theorem 4] {CompositionEntireFunctions} proved that there exist two transcendental entire functions $f$ and $g$ and a domain $U\subset \C$ such that $U$  lies in periodic Fatou components of $f$, $g$, and $g\circ f$, but lies in a wandering domain of $f\circ g$. Sing's result, proved by approximaton theory, does not provide any information about the possible existence of other wandering domains for $f$ or $g$, nor about the global nature of the maps. We give a complete answer to the problem by proving the following theorem.

\begin{thmalpha}\label{thm:B}
There exist two transcendental entire functions $f$ and $g$ in class $\classB$ such that $f\circ g$ has two grand orbits of wandering domains even though all Fatou components of $f$ and $g$ are preperiodic.
\end{thmalpha}

The proof of Theorem B (see Section \ref{sec:thmB}) consists of a subtle modification of Bishop's construction in Theorem \ref{thm:Bishop_example}, and additionally uses also Theorem A. In order to present the proof properly we previously explain Bishop's example (see Section \ref{sec:Bishop_example}), adding some explicit arguments and details to the original construction, in places where we believe they can help the reader to understand the delicate construction.

The structure of the paper is as follows. In Section \ref{sec:preliminaries} we introduce notation, state Bishop's general folding theorem and Rempe-Gillen-Mihaljevi\'c-Brandt's key  lemma, together with some useful results in complex analysis that are used throughout the paper. In Section \ref{sec:Bishop_example} we revise Bishop's wandering example. Finally, theorems A and B are proved in Sections \ref{sec:thmA} and \ref{sec:thmB}, respectively.


\subsection*{Acknowledgements}

We heartfully thank Chris Bishop for his patience and kindness in answering all our questions about his remarkable result. We also wish to thank Lasse Rempe, Phil Rippon, Gwyneth Stallard and Toni Garijo for many helpful discussions.


\section{Preliminaries}\label{sec:preliminaries}


\subsection{Bishop's theorem}\label{sec:Bishop_theorem}

This subsection is mainly devoted to state a slightly simplified version of Bishop's Theorem \ref{thm:Bishop_example} \cite[Theorem 7.1]{QuasiconformalFoldings} which is enough for our purposes. The goal of this theorem is to provide sufficient conditions to guarantee the existence of  entire functions in class $\classB$ with a prescribed distribution of critical points.   We refer the reader to the original source  \cite{QuasiconformalFoldings} for the proof and for a much deeper discussion of its contents, motivation and difficulties. 

As we mentioned in the introduction, Bishop applied this theorem to provide a long list of interesting examples (and counter-examples) of many different phenomena. We shall discuss one of them, concerning the existence of wandering domains in class $\classB$, in plenty of detail  in  Section \ref{sec:Bishop_example}.   

Let $T$ be an unbounded connected bipartite graph with vertice labels in $\{-1,+1\}$. Then the connected components of $\C\setminus T$ are simply connected domains in $\C$. We denote by $R$-components (respectively $D$-components) the unbounded (respectively bounded) components of $\C\setminus T$. We define a neighborhood of the graph given by
$$
T(r):=\bigcup_{e\text{ edge of }T}\Big\{z\in\C\ |\ \dist(z,e)<r\diam(e)\Big\},
$$
where $\dist$ and $\diam$ denote the euclidian distance and diameter respectively. 

\begin{dfn}\label{def:unif_bounded_geom}
We say that $T$ has \emph{uniformly bounded geometry} if there exists an universal constant $M\geqslant 1$ such that
\begin{description}
	\item[(i)] the edges of $T$ are $\mathcal{C}^{2}$ with uniform bounds;
	\item[(ii)] the angles between two adjacent edges are uniformly bounded away from zero;
	\item[(iii)] for every two adjacent edges $e$ and $e'$, $\frac{1}{M}\leqslant\frac{\diam(e)}{\diam(e')}\leqslant M$;
	\item[(iv)] for every two non-adjacent edges $e$ and $e'$, $\frac{\diam(e)}{\dist(e,e')}\leqslant M$.
\end{description}
\end{dfn}

We denote by $\Hr=\{z=x+iy\in\C\,|\,x>0\}$ the right half plane and by $\D=\{z\in\C\,|\,|z|<1\}$ the unit disk. For each connected component $\Omega_{j}$ of $\C\setminus T$, let $\tau_{j}:\Omega_{j}\to \Delta_j$ be the Riemann map where $\Delta_j=\Hr$ or $\Delta_j=\D$ depending on whether $\Omega_{j}$ is an unbounded or bounded component. We shall denote by $\tau$ the global map defined on $\cup_j \Omega_j$ such that $\tau|_{\Omega_j} = \tau_j$. Each edge $e$ of $T$ is the common boundary of at most two complementary domains but corresponds via $\tau$ to exactly two intervals on $\partial \Hr$, or one interval on $\partial \Hr$ and one in $\partial \D$ (see condition (i) in Theorem \ref{thm:Bishop_main}). So it makes sense to define the \textit{$\tau$-size} of an edge $e$ as the minimum among the lengths of the two images of $e$ by $\tau$. 
We call $\Delta=\Delta_j$  the {\em standard domain} for $\Omega_{j}$. Next step is to build a map from the standard domain to $\C$. 

If $\Delta=\D$ we define $\sigma:\D\to\D$ to be the  map $z\to z^{d}$  with $d\geqslant 2$, possibly followed by a quasiconformal map $\rho:\D\to\D$ which sends $0$ to some $w\in\D$  and it is the identity on $\partial\D$. The resulting map $\sigma\circ\tau:\Omega_{j}\to\D$ will have $w$ as a critical value.

If $\Delta=\Hr$ the definition of $\sigma:\Hr\to\C$ is more delicate. We first divide $\partial\Hr$ into intervals $I$ of length $2\pi$ and extremes (or vertices) in $2\pi i\Z$. These intervals will correspond, after Bishop's folding construction, to the images of the edges of $T'\supset T$ (where $T'$ is $T$ plus some \textit{decorations} all of them contained  in a sufficiently small neighborhood $T(r_{0})$ of $T$) by a suitable quasiconformal deformation $\eta$ of $\tau$. Second we define $\sigma$ on $\partial\Hr$. There will be only two cases to consider: either $I$ is identified with a common arc of two $R$-components in which case we define $\sigma(iy):=\cosh(iy)$ for every $iy\in I$, or $I$ is identified with as a common arc of one $R$-component and one $D$-component, in which case we define $\sigma(iy):=\exp(iy)$ for every $iy\in I$. Finally we extend $\sigma$ to $\Hr$ as a quasiconformal map which equals $\cosh(x+iy)$ for $x>2\pi$.

\begin{thm}[{\cite[Theorem 7.1]{QuasiconformalFoldings}}]\label{thm:Bishop_main} Let $T$ be an unbounded connected graph and let $\tau$ be a conformal map defined on each complementary domain $\C\setminus T$, as above. Assume that
\begin{description}
\item[(i)] $D$-components only share edges with $R$-components;
\item[(ii)] $T$ is bipartite and has uniformly bounded geometry;
\item[(iii)] the map $\tau$ on a $D$-component with $2n$ edges maps the vertices to $2n$-th roots of unity;
\item[(iv)]  on $R$-components the $\tau$-sizes of all edges is $\geqslant\pi$ (it is enough that the $\tau$-sizes are uniformly bounded from below).
\end{description}
Then there is $r_{0}>0$, an entire transcendental function $f$, and a $K$-quasiconformal map $\phi$ of the plane, with $K$ depending only in $M$,  so that $f\circ\phi^{-1}=\sigma\circ\tau$ off $T(r_{0})$. Moreover, the only singular values of $f$ are $\pm 1$ (critical values coming from the vertices of $T$) and the critical values assigned by the $D$-components.
\end{thm}

The main difficulty in Bishop's construction is to define $\eta$ (and $T'$) in such a way that $\sigma\circ\eta$ is a quasiregular map, or equivalently, to show that $\sigma\circ\eta$ is continuous across $T'$. Once this is done, the Measurable Riemann mapping Theorem can be applied to obtain $\phi$, which integrates the almost complex structure induced by $\sigma\circ\eta$, so that $f:=\sigma\circ\tau\circ\phi$ is entire and in class $\classB$.

In Section \ref{sec:Bishop_example}, we shall see how Bishop applies and modifies this construction to provide an example of a map in class $\classB$ with wandering domains. 


\subsection{Other tools}

The following results are used in Section \ref{sec:Bishop_example}. The first one is the well-known Koebe's distorsion Theorem.

\begin{thm}[{\cite[Theorem 1.6 and Corollary 1.4]{PommerenkeBook}}]\label{thm:Koebe}
Let $F$ be a univalent function in the unit disk $\D$ such that $F(0)=0$ and $F'(0)=1$. Then  
\begin{description}
\item[(a)] $\forall z\in\D,\quad\dfrac{|z|}{(1+|z|)^{2}}\leqslant|F(z)|\leqslant\dfrac{|z|}{(1-|z|)^{2}}$;
\item[(b)] $\forall z\in\D,\quad\dfrac{1-|z|}{(1+|z|)^{3}}\leqslant|F'(z)|\leqslant\dfrac{1+|z|}{(1-|z|)^{3}}$;
\item[(c)] $\dfrac{1}{4}\leqslant\dist(0,\partial F(\D))\leqslant 1$. In particular $F(\D)$ contains a disk of radius $1/4$.
\end{description}
\end{thm}

As explained above the key idea behind Theorem \ref{thm:Bishop_main} is to obtain the desired entire function $f$ as the composition of a quasiregular map $\sigma\circ\eta$, and the quasiconformal map $\phi$ given by the Riemann mapping Theorem, that is $f:=(\sigma\circ\eta)\circ\phi$. In particular, $f$ and $\sigma\circ\eta$ are not conjugate to each other. As it turns out, we shall have an explicit expression for $\sigma\circ\eta$,  at least in the domains where the relevant dynamics occurs so, in order to have a control on the dynamics of $f$ we need certain control on $\phi$. The results that follow  will be used to show that,  in the cases we are interested in,  we may assume that $\phi$ is arbitrarily close to the identity. 

\begin{dfn}
A measurable set $E\subset\R^{2}$ is said to be $(\varepsilon,h)$-thin if $\varepsilon>0$, $h:[0,\infty)\to[0,1]$ is a decreasing function such that
$$
\forall n\geqslant 1,\ \int_{0}^{+\infty}h(r)r^{n}dr<+\infty,
$$
and
$$
\forall z\in E,\ \area\left(E\cap D(z,1)\right)\leqslant\varepsilon h(|z|)
$$
where $D(z,1)$ is the euclidean disk centered at $z$ of radius 1.
\end{dfn}

Notice that $h(r):=\exp(-ar)$ with $a>0$ satisfies the conditions on $h$.

Roughly speaking the next result states that if $\Phi:\R^{2}\to\R^{2}$ is a $K$-quasiconformal map with dilatation $\mu$ supported on a \textit{small} planar set then we may expect  $\Phi$ to be close to a conformal map in $\mathbb C$, and so, close to the identity after normalization. 

\begin{thm}[Bishop, personal communication]\label{thm:close_to_id}
Suppose $\Phi:\R^{2}\to\R^{2}$ is $K$-quasiconformal and is normalized to fix 0 and 1. Let $E$ be the support of the dilatation of $\Phi$ (possibly unbounded) and assume that $E$ is $(\varepsilon,h)$-thin. Then
\begin{equation}\label{eq:close_to_id}
	\forall z,w\in\R^{2},\ (1-Ce^{\beta})|z-w|-Ce^{\beta}\leqslant|\Phi(z)-\Phi(w)|\leqslant(1+Ce^{\beta})|z-w|+e^{\beta}
\end{equation}
where $C$ and $\beta$ only depend on $K$ and $h$.
\end{thm}

As a consequence of the above result, under suitable conditions on the behaviour and smoothness of $\Phi$ in a neighborhood of the real line, we may easily deduce some bounds on the derivative of the map that we will directly apply later.  

\begin{coro}\label{thm:bounds_derivative}
Suppose $\Phi:\R^{2}\to\R^{2}$ is a $K$-quasiconformal map which is conformal in the strip $S=\{x+iy\in\C\,|\,|y|<1\}$ and satisfies $\Phi(0)=0$, $\Phi(1)=1$ and $\Phi(\R)=\R$. Let $E$ be the support of the dilatation of $\Phi$ and assume that $E$ is $(\varepsilon,h)$-thin. If $\varepsilon$ is sufficiently small (depending on $K$ and $h$), then 
$$
\forall x\in\R,\ \frac{1}{C}\leqslant|\Phi'(x)|\leqslant C
$$ 
where $C>0$ only depends on $K$, $h$ and $\varepsilon$ (but $C$ does not depends on the particular map $\Phi$). Moreover if we fix $K$ and 
$h$ then $C\to 1$ as $\varepsilon\to 0$.
\end{coro}

\begin{proof}
Let $x\in\R$. Since $D(x,1)\subset S$.   From Theorem \ref{thm:Koebe}   applied to $(\Phi(z+x)-\Phi(x))/\Phi'(x)$ we have
$$
|\Phi'(x)|\simeq\dist\Big(\Phi(x),\partial\Phi(D(x,1))\Big).
$$
However taking $z=x$ and $w\in\partial D(x,1)$ in (\ref{eq:close_to_id}), we also have $\dist\left(\Phi(x),\partial\Phi(D(x,1))\right)\simeq 1$. This shows the first claim.

On the other hand, when $\varepsilon$ is small, applying again (\ref{eq:close_to_id}), we have $(1-\delta)S\subset\Phi(S)\subset (1+\delta)S$ where $\delta$ is small and tends to zero with $\varepsilon$. Hence, $\Phi$ converges to the identity in $S$.  
\end{proof}

\begin{rem}
In some later applications the $K$-quasiconformal map $\Phi$ of the previous results is the 
integrator of the quasiregular map $\sigma\circ\eta$ given by the Measurable Riemann mapping Theorem normalized so that $\Phi(0)=0$ and $\Phi^{\prime}(\infty)=1$.   
\end{rem}

We end this section by stating a  lemma which we shall use to prove Theorem A. 

\begin{thm}[{\cite[Theorem 4.1]{AbsenceWanderingDomains}}]\label{thm:Rempe}
Let $U$ be a hyperbolic Riemann surface and denote by $\dist_{U}$ the hyperbolic distance in $U$. Let $U'\subset U$ be open and let $f:U'\rightarrow U$ be a holomorphic covering map. Assume that there is an open connected set $W\subset U'$ such that $f^{n}(W)\subset U'$ for every $n\geqslant 0$. Furthermore, let $D\subset U$ be open and suppose there is a subsequence of positive integers $(n_{k})_{k\geqslant 1}$ such that $f^{n_{k}}(W)\subset D$ for every $k\geqslant 1$ and $\dist_{U}(f^{n_{k}}(W),U\setminus D)\to+\infty$ as $k\to+\infty$. Then
$$
\dist_{U}(f^{n_{k}-1}(W),U\setminus f^{-1}(D))\to+\infty\text{ as $k\to+\infty$}.
$$
\end{thm}


\section{Bishop's example}\label{sec:Bishop_example}

This section is an account  of \cite[Section 17]{QuasiconformalFoldings}, where the author first gives an application of Theorem \ref{thm:Bishop_main} (c.f. {\cite[Theorem 7.1]{QuasiconformalFoldings}}) to construct a family of entire functions in class $\classB$ depending on infinitely many parameters (see Theorem \ref{thm:prototype}). Then he choses suitable parameters in order to ensure that the resulting function has oscillating wandering domains, thus proving Theorem \ref{thm:Bishop_example}. 

We include some details of the proof of this result that were left to the reader in the original paper, since they are important in the proof of Theorems \ref{thm:A} and Theorem \ref{thm:B}. Hence, although there is nothing new from Bishop's paper \cite{QuasiconformalFoldings}, the explicit computations of this section make that the slight modifications in the construction done in Section \ref{sec:thmB} are feasible.


\subsection{The prototype map}

Consider the open half strip
$$
S^{+}:=\left\{x+iy\in\C\ |\ x>0\ \text{and}\ |y|<\frac{\pi}{2}\right\}.
$$
The set $S^{+}$ is mapped biholomorphically onto the right half plane $\Hr=\{z=x+iy\in\C \mid x>0\}$ by $z\mapsto\sinh(z)$, thus by $z\mapsto\lambda\sinh(z)$ as well where $\lambda\in\pi\Nstar$ is a parameter to be fixed later (and $\Nstar=\N\setminus\{0\}$). Then $\Hr$ is mapped conformally (but not biholomorphically)  onto $\C\backslash[-1,+1]$ by $z\mapsto\cosh(z)$. Following the notation in \cite{QuasiconformalFoldings}, we denote by $(\sigma\circ\tau)|_{S^{+}}$ the composition $z\mapsto\cosh(\lambda\sinh(z))$. Remark that this map extends continuously  to the boundary, sending $\partial S^{+}$ onto the imaginary axis and then onto the real segment $[-1,+1]$.

On the horizontal rays of $\partial S^{+}$, we need to pick out some points $(a_{n}\pm i\pi/2)_{n\geqslant 1}$ which are sent into $\{-1,+1\}$ by $(\sigma\circ\tau)|_{S^{+}}$, and so that $a_{n}$ is close to $n\pi$ for every $n\geqslant 1$. More precisely, the real parts $(a_{n})_{n\geqslant 1}$ are defined as follows:
$$
\forall n\geqslant 1,\quad a_{n}:=\cosh^{-1}\left(\frac{\pi}{\lambda}\left[\frac{\lambda}{\pi}\cosh(n\pi)\right]\right)
$$
where $[x]$ stands for the integer part of any real number $x$. It is straightforward to show that $ \cosh(\lambda\sinh(a_{n}\pm i\pi/2))\in\{-1,+1\}$ and $n\pi\geqslant a_{n}>n\pi-10^{-1}$ for every $n\geqslant 1$ (because $\lambda\in\pi\Nstar$).

Now, consider the following open disks:
$$
\forall n\geqslant 1,\quad D_{n}:=\{z\in\C\ |\ |z-z_{n}|<1\}\quad\text{where}\quad z_{n}:=a_{n}+i\pi.
$$
Notice that these domains do not intersect $S^{+}$, and are pairwise disjoint according to the definition of $(a_{n})_{n\geqslant 1}$. Every domain $D_{n}$ is biholomorphically mapped onto the unit disk $\D=\{z\in\C\,|\,|z|<1\}$ by $z\mapsto z-z_{n}$. Define a quasiregular map from $\D$ onto itself of the form $z\mapsto\rho_{n}(z^{d_{n}})$ where $d_{n}\in 2\Nstar$ is a parameter to be fixed later and $\rho_{n}:\overline{\D}\rightarrow\overline{\D}$ is a quasiconformal map so that
\begin{itemize}
	\item $\rho_{n}(z)=z$ for every $z\in\partial\D$;
	\item $\rho_{n}(0)=w_{n}$ where $w_{n}$ is a parameter to be fixed later in a neighborhood $\neighbor_{1/2}$ of $1/2$ which does not depend on $n$ and is fixed so that $\overline{\neighbor_{1/2}}\subset\D$;
	\item $\rho_{n}$ is conformal on $\frac{3}{4}\D$;
	\item $\rho_{n}$ is $K_{\rho}$-quasiconformal on $\D$ where the dilation $K_{\rho}>1$ does not depend on $n$.
\end{itemize}
To fix  ideas, we may define $\rho_{n}$ to be $\rho_{n}(z):=\delta z+w_{n}$ on $\overline{D(0,3/4)}=\{z\in\C\,|\,|z|\leqslant 3/4\}$, and to be the linear interpolation $\rho_{n}(z):=z(4|z|-3)+(\delta z+w_{n})(4-4|z|)$ on $\overline{A_{3/4}}=\{z\in\C\,|\,3/4\leqslant|z|\leqslant 1\}$, where $\delta=\dist(\neighbor_{1/2},\partial\D)>0$ in order that $\rho_{n}(\overline{D(0,3/4)})=\overline{D(w_{n},3\delta/4)}\subset\D$. Since the Beltrami coefficient of $\rho_{n}$ is supported on $A_{3/4}$, and the modulus of the annulus $\rho_{n}(A_{3/4})=\{z\in\C\,|\,3\delta/4<|z-w_{n}|<1\}$ is uniformly bounded from above and below, it is straightforward to show that the dilatation of $\rho_{n}$ is uniformly bounded from above by a universal constant $K_{\rho}>1$ which only depends on $\delta$.

Following the notation in \cite{QuasiconformalFoldings}, we denote by $(\sigma\circ\tau)|_{D_{n}}$ the composition $z\mapsto\rho_{n}((z-z_{n})^{d_{n}})$ for every $n\geqslant 1$. Notice that 
the dilatation of $(\sigma\circ\tau)|_{D_{n}}$ is uniformly bounded from above by $K_{\rho}$ and it is supported on
\begin{equation}\label{eq:support}
	\left\{z\in\C\ |\ \left(\frac{3}{4}\right)^{1/d_{n}}<|z-z_{n}|<1\right\}.
\end{equation}
Figure \ref{fig:domains} summarizes the whole construction above.

\begin{figure}[!ht]
	\begin{center}
	\def\svgwidth{\textwidth}
	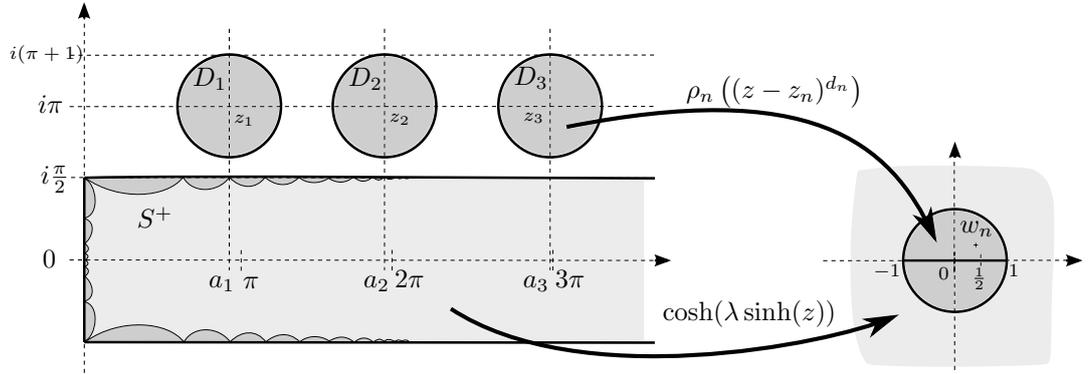
	\caption{\small The domains $S^{+}$ and $(D_{n})_{n\geqslant 1}$ are depicted on the left. The dark gray areas stand for the preimages of the unit disk $\D$ under the map $\sigma\circ\tau$.}\label{fig:domains}
	\end{center}
\end{figure}

The following result states that we can extend the maps $(\sigma\circ\tau)|_{S^{+}}$ and $(\sigma\circ\tau)|_{D_{n}}$  to a quasiregular map defined on the whole complex plane.  That provides an entire function $f$ after composing with the integrating map of the pullback of the standard complex structure. The function $f$ is actually the prototype map  which, for a certain choice of parameters $(\lambda,(d_{n})_{n\geqslant 1},(w_{n})_{n\geqslant 1})$,  has an oscillating wandering domain.

\begin{thm}\label{thm:prototype}
For every choice of the parameters $(\lambda,(d_{n})_{n\geqslant 1},(w_{n})_{n\geqslant 1})$ such that $\lambda\in\pi\Nstar$, $d_{n}\in 2\Nstar$, and $w_{n}\in\neighbor_{1/2}$ for all $n\geqslant 1$, there exists a transcendental entire function $f$ and a quasiconformal map $\phi:\C\rightarrow\C$ such that
\begin{description}
	\item[(a)] for every $z\in\C$, $f(\bar{z})=\overline{f(z)}$ and $f(-z)=f(z)$;
	\item[(b)] $f\circ\phi^{-1}$ extends the maps $(\sigma\circ\tau)|_{S^{+}}$ and $(\sigma\circ\tau)|_{D_{n}}$ for every $n\geqslant 1$, or equivalently
	$$f(z)=\left\{\begin{array}{rcl}
		\cosh(\lambda\sinh(\phi(z))) & \text{if} & \phi(z)\in S^{+} \\
		\rho_{n}\left((\phi(z)-z_{n})^{d_{n}}\right) & \text{if} & \phi(z)\in D_{n} \\
\end{array}\right.;$$
	\item[(c)] $f$ has no asymptotic values, and its set of critical values is $\{-1,+1\}\cup\overline{\{w_{n}\,/\,n\geqslant 1\}}$ (hence  $f$ is in class $\classB$);
	\item[(d)] $\phi(0)=0$, $\phi(\R)=\R$, $\phi$ is conformal in $S^+$ and its dilatation is uniformly bounded from above by an universal constant $K>1$ which does not depend on the parameters.  
\end{description}
\end{thm}


\subsection{Proof of Theorem \ref{thm:prototype}}

In this subsection, we give the proof of Theorem \ref{thm:prototype}. The main idea is   to apply Theorem \ref{thm:Bishop_main}, that is to define a graph $T$ and the corresponding maps $\tau_{j}$ in each connected component of the complement of $T$ and show that they are under the hypotheses of the theorem. This discussion is outlined in \cite[Section 17]{QuasiconformalFoldings}. 

 
The graph $T$, yet to be defined,  will be symmetric with respect to both the real and imaginary axis. The boundaries of the sets $S^{+}$ and $(D_{n})_{n\geqslant 1}$ introduced in the previous section will be part of $T$, hence the sets themselves will be connected components of $\C\setminus T$.

We connect every $\partial D_{n}\subset T$ to $\partial S^{+}\subset T$  by the vertical segment joining $z_{n}-i=a_{n}+i(\pi-1)\in\partial D_{n}$ to $a_{n}+i\pi/2\in\partial S^{+}$. We also add to $T$ the vertical rays starting at every vertex $z_{n}+i\in\partial D_{n}$ and going straight up to infinity, together with the imaginary axis. That produces vertical unbounded connected components of $\C\backslash T$ (that are almost half strips except for the indentations caused by the disks $(D_{n})_{n\geqslant 1}$). For every $n\geqslant 1$, we denote by $S_{n}$ every such domain located between $D_{n}$ and $D_{n+1}$, and by $S_{0}$ the domain located between the imaginary axis and $D_{1}$. Finally, we complete the  remainder of the graph $T$ by symmetry with respect to both the real and imaginary axes. The graph $T$ is shown in Figure \ref{fig:graph}. Clearly condition \textbf{(i)} of Theorem \ref{thm:Bishop_main} follows.

\begin{figure}[!ht]
	\begin{center}
	\def\svgwidth{0.75\textwidth}
	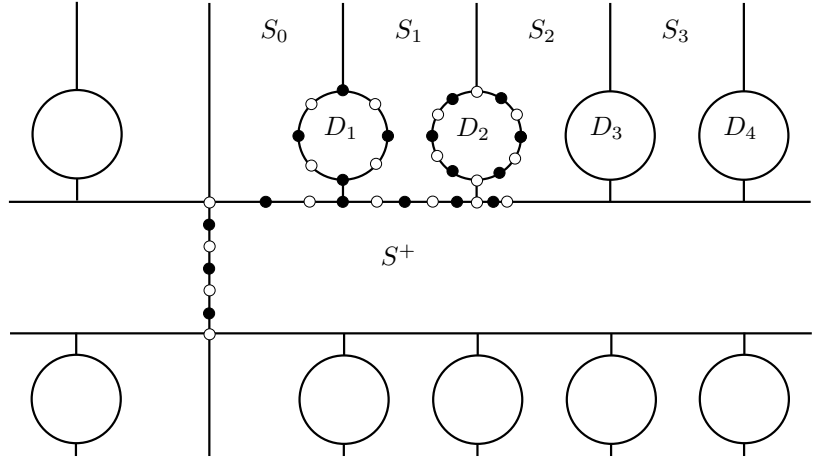
	\caption{The graph $T$ with some first labelled vertices.}\label{fig:graph}
	\end{center}
\end{figure}

Next step is to show that $T$ has uniformly bounded geometry (see Definition \ref{def:unif_bounded_geom}). Items \textbf{(i)} and $\textbf{(ii)}$ in Definition \ref{def:unif_bounded_geom} hold since every edge is either a vertical segment, a horizontal segment, or a circular arc of radius $1$.

On $\partial S^{+}\subset T$, we define the vertices of $T$ to be the points which are sent into $\{-1,+1\}$ by the composition $(\sigma\circ\tau)|_{S^{+}}:z\mapsto\cosh(\lambda\sinh(z))$. These vertices are of the form:
$$
\left\{i\sin^{-1}\left(\frac{\pi}{\lambda}k\right)\ /\ k\in\Z\ \text{and}\ -\frac{\lambda}{\pi}\leqslant k\leqslant\frac{\lambda}{\pi}\right\}\cup\left\{\cosh^{-1}\left(\frac{\pi}{\lambda}k\right)\pm i\frac{\pi}{2}\ /\ k\in\Z\ \text{and}\ k\geqslant\frac{\lambda}{\pi}\right\}.
$$
In particular, notice that the branching points $\pm i\pi/2$ are vertices of $T$ since $\lambda\in\pi\Nstar$, and the branching points $(a_{n}\pm i\pi/2)_{n\geqslant 1}$ are vertices of $T$ as well. By using classical properties of the maps $\sin^{-1}$ and $\cosh^{-1}$, it is straightforward to show the following estimates 
$$\begin{array}{rccccc}
	\forall\quad 2\leqslant k\leqslant\dfrac{\lambda}{\pi}, & 1 & \leqslant & \dfrac{\sin^{-1}\left(\frac{\pi}{\lambda}k\right)-\sin^{-1}\left(\frac{\pi}{\lambda}(k-1)\right)}{\sin^{-1}\left(\frac{\pi}{\lambda}(k-1)\right)-\sin^{-1}\left(\frac{\pi}{\lambda}(k-2)\right)} & \leqslant & \dfrac{\sqrt{2}}{2-\sqrt{2}} \\
	\forall\quad k\geqslant\dfrac{\lambda}{\pi}, & 1 & \geqslant & \dfrac{\cosh^{-1}\left(\frac{\pi}{\lambda}(k+2)\right)-\cosh^{-1}\left(\frac{\pi}{\lambda}(k+1)\right)}{\cosh^{-1}\left(\frac{\pi}{\lambda}(k+1)\right)-\cosh^{-1}\left(\frac{\pi}{\lambda}k\right)} & \geqslant & \dfrac{\cosh^{-1}(3)}{\cosh^{-1}(2)}-1 \\
	& 1 & \leqslant & \dfrac{\sin^{-1}\left(\frac{\pi}{\lambda}\frac{\lambda}{\pi}\right)-\sin^{-1}\left(\frac{\pi}{\lambda}(\frac{\lambda}{\pi}-1)\right)}{\cosh^{-1}\left(\frac{\pi}{\lambda}(\frac{\lambda}{\pi}+1)\right)-\cosh^{-1}\left(\frac{\pi}{\lambda}\frac{\lambda}{\pi}\right)} & \leqslant & \dfrac{\pi/2}{\cosh^{-1}(2)} .
\end{array}$$

From these estimates we obtain items \textbf{(iii)} and \textbf{(iv)} in Definition \ref{def:unif_bounded_geom} for every pair of edges in $\partial S^{+}$.

Similarly, on every $\partial D_{n}\subset T$, we define the vertices of $T$ to be the points which are sent into $\{-1,+1\}$ by the composition $(\sigma\circ\tau)|_{D_{n}}:z\mapsto\rho_{n}((z-z_{n})^{d_{n}})$. Taking into account that $\rho_{n}$ is the identity map on $\partial\D$, these vertices are precisely the translated $2d_{n}$-th roots of unity. Notice that the branching points $z_{n}\pm i$ are vertices of $T$ since $d_{n}$ is even. Moreover,  items \textbf{(iii)} and \textbf{(iv)} in Definition \ref{def:unif_bounded_geom} hold for every pair of edges in $\partial D_{n}$ because they have equal diameter less than $2\pi/2d_{n}\leqslant\pi$.

We also label every vertex defined above by either $-1$ or $+1$ according to its image by the composition $\sigma\circ\tau$.

Now, consider the vertical segment in $T$ which connects $\partial D_{n}$ to $\partial S^{+}$. It has constant length equal to $\pi/2-1$ while the edges connected at both ends, which are $z_{n}-i=a_{n}+i(\pi-1)$ and $a_{n}+i\pi/2$, have lengths comparable to $\ell_{n}=2\pi/2d_{n}$ and $L_{n}=\cosh^{-1}(\cosh(n\pi)+\pi/\lambda)-n\pi$ respectively (because $a_{n}$ is close to $n\pi$). If $\ell_{n}$ and $L_{n}$ are comparable, it is enough to subdivide the segment into finitely many edges of comparable lengths, taking care that two adjacent additional vertices have different label in $\{-1,+1\}$ (adding one extra edge if necessary). So we may assume, without loss of generality, that $\ell_{n}=o(L_{n})$. The idea is to introduce additional labelled vertices in order that the lengths of the first edges from $\partial D_{n}$, namely from $a_{n}+i(\pi-1)$, are in a geometric progression of uniformly constant ratio, say $2$. More precisely, we define these additional vertices to be as follows
$$\begin{array}{l}
	a_{n}+i(\pi-1-\ell_{n}) \\
	a_{n}+i(\pi-1-\ell_{n}-2\ell_{n}) \\
	a_{n}+i(\pi-1-\ell_{n}-2\ell_{n}-4\ell_{n}) \\
	\dots \\
	a_{n}+i(\pi-1-\ell_{n}-2\ell_{n}-4\ell_{n}-\dots-2^{k}\ell_{n}) \\
	\dots \\
\end{array}$$
and we end this geometric progression after finitely many edges as soon as the length of the last edge, which is of the form $2^{k}\ell_{n}$, is comparable to $L_{n}$. The assumption $\ell_{n}=o(L_{n})$ ensures that the additional vertices so defined are not outside the vertical segment of constant length joining $D_{n}$ and $S^{+}$. Then, we subdivide the residual segment into finitely many edges of lengths comparable with $L_{n}$, taking care that two adjacent additional vertices have different label in $\{-1,+1\}$ (adding one extra edge if necessary). By construction, these edges satisfy items $\textbf{(iii)}$ and \textbf{(iv)} in Definition \ref{def:unif_bounded_geom}.

On the vertical rays starting at every vertex $z_{n}+i$ and going straight up to infinity, it is enough to add infinitely many additional labelled vertices, spaced with constant length equal to $2\pi/2d_{n}$, ensuring that two adjacent additional vertices have different label in $\{-1,+1\}$. For the vertical ray in the imaginary axis, starting at $i\pi/2$ and going straight up to infinity, we proceed as well with labelled vertices separated by constant distance equal to $\pi/2-\sin^{-1}(1-\pi/\lambda)$. Finally, we complete the set of vertices of $T$ by symmetry with respect to both the real and imaginary axes. Therefore, the unbounded connected graph $T$ is bipartite and has uniformly bounded geometry. So, condition \textbf{(ii)} in Theorem \ref{thm:Bishop_main} is also satisfied.

The third condition in Theorem \ref{thm:Bishop_main} is satisfied by construction. It remains to check  condition \textbf{(iv)}, which  requires that every connected component $\Omega$ of $\C\setminus T$ is  mapped biholomorphically onto either $\Hr$ or $\D$, by a map $\tau|_{\Omega}$ such that $\tau|_{\Omega}^{-1}$  extends continuously to the boundary, and each connected component of the preimage under $\tau|_{\Omega}^{-1}$ of every edge in $T$ (called a $\tau$-image) has length $\geqslant\pi$.

Recall that $\tau|_{S^{+}}:z\mapsto\lambda\sinh(z)$  maps $S^{+}$ onto $\Hr$ biholomorphically,  extends continuously to the boundary, and sends every edge in $\partial S^{+}\subset T$ onto a vertical segment of length $\pi$ in $\partial\Hr$. Similarly, every $\tau|_{D_{n}}:z\mapsto z-z_{n}$  maps $D_{n}$ onto $\D$ biholomorphically,  extends continuously to the boundary, and sends every edge in $\partial D_{n}\subset T$ onto a half circle of length $\pi$ in $\partial\D$.

Now fix $n\geqslant 0$ and  consider the  almost half strip $S_{n}$. Notice that $\diam(e)$ is uniformly bounded away from $0$ for every edge $e$ in $\partial S_{n}\subset T$ provided that $n$ is fixed. The domain $S_{n}$ is mapped onto $S^{+}$ by a biholomorphism, say $\phi_{n}:S_{n}\rightarrow S^{+}$, which  extends continuously to the boundary. Since $S_{n}$ differs from a straight half strip only in a compact set of the complex plane, we can show that $z\mapsto\phi_{n}(z)$ tends uniformly to an affine transformation, as the imaginary part of $z$ tends  to $+\infty$ (for instance by using the estimations in \cite{ConformalMappingInfiniteStrips}). In particular $\diam(\phi_{n}(e))$ is uniformly bounded away from $0$ for every edge $e$ in $\partial S_{n}\subset T$. By using  properties of the map $\sinh$, the same holds as well for $\diam(\sinh(\phi_{n}(e)))$ where the composition $z\mapsto\sinh(\phi_{n}(z))$ maps $S_{n}$ onto $\Hr$ biholomorphically. Therefore, there exists a positive real number $\lambda_{n}$ large enough so that the map $\tau|_{S_{n}}:z\mapsto\lambda_{n}\sinh(\phi_{n}(z))$ satisfies the assumption.

At this point we may apply Theorem \ref{thm:Bishop_main} to conclude statements \textbf{(a)-(d)} in Theorem \ref{thm:prototype}.  Notice that no quasiconformal foldings occur in $S^{+}$ and $(D_{n})_{n\geqslant 1}$ (because the length of every $\tau$-image is exactly $\pi$), and the dilatation of $\phi$ is uniformly bounded  above by a universal constant $K\geqslant 1$, which only depends on the upper bound $K_{\rho}\geqslant 1$ for the dilation of every $\rho_{n}$, and on the universal constant $M\geqslant 1$ which appears in the geometrical properties of $T$.


\subsection{Proof of Theorem \ref{thm:Bishop_example}}

Following \cite[Section 17]{QuasiconformalFoldings}, we explain in this subsection how to choose the parameters in Theorem \ref{thm:prototype} in order to obtain an entire function in class $\classB$ with an oscillating wandering domain, thus proving Theorem \ref{thm:Bishop_example}. The results here are  very relevant for the proof of Theorem \ref{thm:A} and especially of Theorem \ref{thm:B}.

Let $(\lambda,(d_{n})_{n\geqslant 1},(w_{n})_{n\geqslant 1})$ be some parameters satisfying \textbf{(a)-(d)} in Theorem \ref{thm:prototype}. Remark that the orbit of $\frac{1}{2}$ under iteration of $f$ stays in the positive real line by \textbf{(b)} and \textbf{(d)}. The following result states that, under a certain condition,  this orbit actually escapes to infinity. That allows us to construct infinitely many tiny disks very close to $1/2$ which follow its orbit for arbitrary many iterates before they enter in some of the disks $(D_{n})_{n\geqslant 1}$.

\begin{lem}\label{lem:construction_Un}
Let $f$ and $\phi$ as in  Theorem \ref{thm:prototype} for any choice of parameters $(\lambda,(d_{n})_{n\geqslant 1},$ $(w_{n})_{n\geqslant 1})$. If the following condition holds:
\begin{equation}\label{eq:parameters_assumption}
	\forall x\geqslant 0,\quad\frac{d\phi}{dx}(x)\geqslant\frac{10}{\lambda},
\end{equation}
then the orbit of $1/2$ under iteration of $f$ escapes to infinity, and there exists a sequence of euclidean disks $(U_{n})_{n\geqslant 1}$, together with a subsequence of positive integers $(p_{n})_{n\geqslant 1}$ such that for every $n\geqslant 1$:
\begin{description}
	\item[(a)] $U_{n}$ has radius $0.009(\frac{d}{dx}f^{n}(1/2))^{-1}$ with $(\frac{d}{dx}f^{n}(1/2))^{-1}\leqslant 50^{-n}/n!$;
	\item[(b)] $U_{n}$ is contained in the disk centered at $1/2$ and of radius $20(\frac{d}{dx}f^{n}(1/2))^{-1}$;
	\item[(c)] $f^{k}(U_{n})\subset S^{+}$ for every $0\leqslant k\leqslant n-1$, and
	\item[(d)] $f^{n}(U_{n})\subset\frac{1}{4}\widetilde{D_{n}}$ where $\frac{1}{4}\widetilde{D_{n}}:=\{z\in\C\,|\,|z-z_{p_{n}}|\leqslant 1/4\}$.
\end{description}
\end{lem}

\begin{proof}
Denote by $(x_{k}=f^{k}(1/2))_{k\geqslant 0}$ the orbit of $1/2$ under iteration of $f$. We claim that if  condition (\ref{eq:parameters_assumption}) holds then 
\begin{equation}\label{eq:escaping}
	\forall k\geqslant 0,\quad x_{k+1}-x_{k}\geqslant 11\quad\text{and}\quad\frac{d}{dx}f(x_{k})\geqslant 50.
\end{equation}
Indeed, from  statement \textbf{(b)} in Theorem \ref{thm:prototype}, we get for every $x\geqslant 0$
\begin{eqnarray*}
	\frac{d}{dx}f(x) & = & \frac{d}{dx}\cosh(\lambda\sinh(\phi(x))) \\
	& = & \sinh(\lambda\sinh(\phi(x)))\lambda\cosh(\phi(x))\frac{d}{dx}\phi(x) \\
	& \geqslant & \lambda\phi(x)\lambda\frac{d}{dx}\phi(x)
\end{eqnarray*}
by using the facts that $\sinh(t)\geqslant t$ and $\cosh(t)\geqslant 1$ for every $t\geqslant 0$. Furthermore, integrating  condition (\ref{eq:parameters_assumption}) gives $\phi(x)\geqslant\frac{10}{\lambda}x$ (since $\phi(0)=0$ from statement \textbf{(d)} in Theorem \ref{thm:prototype}). Therefore
$$
\frac{d}{dx}f(x)\geqslant 100 \,x\quad\text{and}\quad f(x)\geqslant 50 \,x^{2}-f(0)=50\,x^{2}-1.
$$
In particular, the orbit $(x_{k})_{k\geqslant 0}$ of $1/2$ escapes to infinity, $\frac{d}{dx}f^{n}(1/2)\geqslant 50^{n}n!$ for every $n\geqslant 1$ by using the chain rule, and it is straightforward to prove  (\ref{eq:escaping}).

Define every $p_{n}$ so that $|x_{n}-a_{p_{n}}|$ is minimal. According to the definition of $(a_{n})_{n\geqslant 1}$, $(p_{n})_{n\geqslant 1}$ is  a strictly increasing sequence of positive integers.

Now fix $n\geqslant 1$, and remark that $\frac{1}{4}\widetilde{D_{n}}:=D(a_{p_{n}}+i\pi,1/4)\subset D(x_{n},5)$ since $|x_{n}-a_{p_{n}}|\leqslant(\pi+10^{-1})/2$. It follows from the first estimate in (\ref{eq:escaping}) that $D(x_{n},10)$ does not intersect $\overline{\D}$, and hence contains no critical values of $f$ according to \textbf{(c)} in Theorem \ref{thm:prototype}. Therefore $f$ has a univalent inverse branch on $D(x_{n},10)$ that maps $D(x_{n},5)$ onto a neighborhood of $x_{n-1}$. By using \textbf{(a)} from Theorem \ref{thm:Koebe}, this neighborhood is contained in the disk centered at $x_{n-1}$  of radius
$$
10\frac{5/10}{(1-5/10)^{2}}\left(\frac{d}{dx}f^{-1}(x_{n})\right)=20\left(\frac{d}{dx}f(x_{n-1})\right)^{-1}
$$
which is less than $\pi/2$ according to the second estimate in (\ref{eq:escaping}). In particular, $\frac{1}{4}\widetilde{D_{n}}$ has a preimage under $f$ in $S^{+}\cap D(x_{n-1},5)$.

Repeating this process $n$ times gives a preimage of $\frac{1}{4}\widetilde{D_{n}}$ under $f^{n}$ close to $x_{0}=1/2$. By using \textbf{(a)} from Theorem \ref{thm:Koebe}, this preimage is contained in a disk centered at $1/2$  of radius
$$
10\frac{5/10}{(1-5/10)^{2}}\left(\frac{d}{dx}(f^{n})^{-1}(x_{n})\right)=20\left(\frac{d}{dx}f^{n}\left(\frac{1}{2}\right)\right)^{-1}.
$$
Moreover, using \textbf{(c)} and then \textbf{(b)} from Theorem \ref{thm:Koebe}, this preimage contains a disk of radius
$$
\frac{1}{4}\cdot\frac{1}{4}\left(\frac{d}{dz}(f^{n})^{-1}(z_{p_{n}})\right)\geqslant\frac{1}{16}\cdot\frac{1-5/10}{(1+5/10)^{3}}\left(\frac{d}{dx}(f^{n})^{-1}(x_{n})\right)\geqslant 0.009\left(\frac{d}{dx}f^{n}\left(\frac{1}{2}\right)\right)^{-1}.
$$
\end{proof}

\begin{lem}\label{lem:lambda}
There exist a positive real number $\lambda^{0}\in\pi\Nstar$ and an exponentially increasing sequence $(d_{n}^{0})_{n\geqslant 1}$ in $2\Nstar$ such that for every choice of parameters $(\lambda,(d_{n})_{n\geqslant 1},$ $(w_{n})_{n\geqslant 1})$ with $\lambda\geqslant\lambda^{0}$ and $d_{n}\geqslant d_{n}^{0}$ for every $n\geqslant 1$, condition (\ref{eq:parameters_assumption}) in Lemma \ref{lem:construction_Un} holds, and hence the euclidean disks $(U_{n})_{n\geqslant 1}$ exist.  
\end{lem}

\begin{proof}
The proof is a direct consequence of Theorem \ref{thm:close_to_id}. Let $\varepsilon:=1$ and $h(x):=\exp(-x)$. It follows from the definition of $T(r)$ (see Subsection \ref{sec:Bishop_theorem}) and the construction of the graph $T$ in the previous subsection that as $\lambda$ and the $(d_{n})_{n\geqslant 1}$ increase, the area of the neighborhood $T(r_{0})$ of the graph $T$ decreases. Indeed, on the one hand, as $\lambda$ increases, the vertices of the graph on the boundary of $S^{+}$ get exponentially close to each other (because of the definition of $\tau(z)=\lambda\sinh(z)$ on $S^{+}$). On the other hand if the $(d_{n})_{n\geqslant 1}$ increase, the vertices on the boundary of the $(D_{n})_{n\geqslant 1}$ also get close to each other and the support of the dilatation of the map $\sigma\circ\tau$ given by (\ref{eq:support}) is getting exponentially small with the $(d_{n})_{n\geqslant 1}$.

Taking this into account we claim that there exist $\lambda^{0}\in\pi\Nstar$ and $\alpha>0$ such that if $\lambda>\lambda^{0}$ and $d_{n}\geqslant d_{n}^{0}:=2[\exp(\alpha n)]$ then $T(r_{0})$ (the support of the dilatation of $\phi$) satisfies 
\begin{equation}
	\area\left(T(r_{0})\cap D(z,1)\right)\leqslant\exp(-R),
\end{equation}
no matter the value of the $(w_{n})_{n\geqslant 1}$ in $\neighbor_{1/2}$. So, the dilatation of $\phi$ occurs on a set which is $(1,h)$-thin. Moreover we also have that $\phi$ is conformal in $S^+$. Consequently Theorem \ref{thm:close_to_id} implies that $|\phi'(x)|\geqslant 1/C$ for all $x\in\R$ with $C$ not depending on the parameters. Up to replacing $\lambda^{0}$ by a real number in $\pi\Nstar$ larger than $\max\{\lambda^{0},10C\}$,  condition (\ref{eq:parameters_assumption}) is satisfied. 
\end{proof}

Let $\left(\lambda^{0},(d_{n}^{0})_{n\geqslant 1},(w_{n}^{0}:=1/2)_{n\geqslant 1}\right)$ be a choice of parameters coming from Lemma \ref{lem:lambda}. We are going to modify the parameters $((d_{n}^0)_{n\geqslant 1},(w_{n}^0)_{n\geqslant 1})$ but $\lambda^{0}$ will be fixed from now on. We consider the $(n+1)$-th iterate of every $U_{n}$ by $f$. Statement \textbf{(d)} of Lemma \ref{lem:construction_Un} gives $f^{n}(U_{n})\subset\frac{1}{4}\widetilde{D_{n}}$. If follows from  Theorem \ref{thm:close_to_id} that $\phi$ is close to the identity, so we have $\phi(\frac{1}{4}\widetilde{D_{n}})\subset\frac{1}{2}\widetilde{D_{n}}$. Then it follows from \textbf{(b)} in Theorem \ref{thm:prototype}, and from the definition of the quasiconformal map $\rho_{n}$, that $f^{n+1}(U_{n})\subset D(w_{p_{n}}^{0},(1/2)^{d_{p_{n}}})$ (by applying Schwarz lemma to $z\mapsto\rho_{n}(z)-w_{n}^{0}$). We are going to adjust $d_{p_{n}}^{0}$ and $w_{p_{n}}^{0}$ in order that $f^{n+1}(U_{n})\subset U_{n+1}$. For this, we need the two following lemmas. The first one deals with the $(d_{p_{n}})_{n\geqslant 1}$ while the second one deals with the $(w_{p_{n}})_{n\geqslant 1}$ (and modify, if necessary, the $(d_{p_{n}})_{n\geqslant 1}$ again).

\begin{lem}\label{lem:dn}
Let $\left(\lambda^{0},(d_{n}^{0})_{n\geqslant 1},(w_{n}^{0}:=1/2)_{n\geqslant 1}\right)$. Then there exists a sequence of positive real numbers $(r_{n})_{n\geqslant 1}$ such that for every new choice of parameters $(d_{n})_{n\geqslant 1}$ with $d_{n}\geqslant d_{n}^{0}$ for every $n\geqslant 1$, the corresponding maps $f$ and $\phi$ satisfy  condition {\rm (\ref{eq:parameters_assumption})} as well, and
$$
\forall n\geqslant 1,\quad 0.009\left(\frac{d}{dx}f^{n}\left(\frac{1}{2}\right)\right)^{-1}\geqslant r_{n}.
$$
In particular, we may assume that for all those parameters, every euclidean disk $U_{n}$ in Lemma \ref{lem:construction_Un} has radius larger or equal than $r_{n}$, and consequently we may choose $d_{p_n}$ such that 
\begin{equation}\label{eq:dpn}
	\left(\frac{1}{2}\right)^{d_{p_{n}}}<r_{n+1}.
\end{equation}
\end{lem}

\begin{proof}
For every $n\geqslant 1$, let $(d_{n}^{k})_{j\geqslant 1}$ be any increasing sequence of positive integers so that $d_{n}^{j}\geqslant d_{n}^{0}$ for all $j\geqslant 1$. The dilation of the corresponding integrating function $\phi_{j}$ (we write $\phi_{j}$ to denote the integrating function corresponding to the sequence $(d_{n}^{j})_{n\geqslant 1}$) remains uniformly bounded by $K$ (see \textbf{(d)} in Theorem \ref{thm:prototype}) while, again from construction, its the support is getting smaller with $j$. Consequently we can take a subsequence $(\phi_{j_{\ell}})_{\ell\geqslant 1}$ that converges uniformly to a $K$-quasiconformal map $\Phi$ in compact subsets of $\C$. Since $f_{j}=(\sigma\circ\eta)_{j}\circ\phi_{j}$ with $(\sigma\circ\eta)_{j}(z)=\cosh(\lambda\sinh(z))$ in $S^{+}$ we also have that $(f_{j_{\ell}})_{\ell\geqslant 1}$ converges in compact subsets of $S^{+}$. In particular for all possible choices of sequence $(d_{n}^{j})_{j\geqslant 1}$ the $(r_{n}^{j})_{j\geqslant 1}$ have a positive lower bound for each $n\geqslant 1$. Hence we may chose a sequence $(d_{n})_{n\geqslant 1}$ so that $(d_{p_{n}})_{n\geqslant 1}$ satisfies (\ref{eq:dpn}).
\end{proof}

We assume $U_{n}\subset\neighbor_{1/2}$ for all $n$ (if necessary we just consider $n$ large enough so that this happens). From (\ref{eq:dpn}) we have chosen the $d_{p_n}$ in such a way that the open disk $U_{n}$ is mapped under $f^{n+1}$ inside an euclidean disk centered at $1/2$ and of radius strictly less than $r_{n+1}$, which in turn is strictly less than the diameter of $U_{n+1}$. Let $w'_{n}$ be the center of the disk $U_{n}$. 

To finish the proof of Theorem \ref{thm:Bishop_example} we should modify $f$ so that 
$$
f^{n+1}(U_{n})\subset D(w'_{n+1},(1/2)^{d_{p_{n}}})\subset 
D(w'_{n+1},r_{n+1})\subset U_{n+1}.
$$
This is proved recursively in the following lemma.

\begin{lem}\label{lem:Bishop_example}
There exists a transcendental entire function $f$ in class $\classB$ coming from Theorem \ref{thm:prototype} for a choice of the parameters $(\lambda,(d_{n})_{n\geqslant 1},(w_{n})_{\geqslant 1})$ which satisfies  condition {\rm (\ref{eq:parameters_assumption})} in Lemma \ref{lem:construction_Un}, such that for every $n\geqslant N$ large enough
$$
f^{n+1}(U_{n})\subset U_{n+1}.
$$
In particular, the euclidean disks $(U_{n})_{n\geqslant N}$, for $N$ large enough,   lie in wandering Fatou domains for $f$.
\end{lem}

\begin{proof}
Let $\left(\lambda^{0},(d_{n})_{n\geqslant 1},(w_{n}^{0}:=1/2)_{n\geqslant 1}\right)$ be a choice of parameters given by Lemma \ref{lem:dn}. We will adjust the parameters recursively. Assume we are in  step $n\geqslant 1$. We want to  modify $\rho_{p_{n}}$ so that it sends $0$ to $w'_{n+1}$ instead of sending $0$ to $w_{p_{n}}^{0}=1/2$ (see \textbf{(b)} of Theorem \ref{thm:prototype}). Doing this, it is clear from Lemma \ref{lem:dn} that $f(\frac{1}{4}\widetilde{D_{n+1}})\subset U_{n+1}$ but we need to be sure that still $f^{n+1}(U_{n+1})\subset\frac{1}{4}\widetilde{D_{n+1}}$. 

Of course the introduced modification caused a correction on the definition of $\phi$ (and hence in the definition of $f$). However the \textit{new complex dilatation} associated to this slight modification of $\rho_{p_{n}}$ is concentrated on an annulus of area as small as we wish by considering, if necessary, a larger $d_{p_{n}}$ (see (\ref{eq:support})). Using again Theorem \ref{thm:close_to_id} this implies that $\phi$ is as close as we wish to the identity in the whole plane; in particular the correction is less than $t_{n}\ll 1$ in every disk of radius 1 centered at $\{x_{0},x_{1},\dots,x_{n+1}\}$. 

Since $t_{n}$ is as small as we wish and the definition of $f$ has not changed in $S^{+}$, we have that still $f^{n+1}(U_{n+1})\subset\frac{1}{4}\widetilde{D_{n+1}}$ for the \textit{new} $f$. Notice that at this step of the construction we cannot guarantee that (the new) $f$ satisfies statement \textbf{(d)} in Lemma \ref{lem:construction_Un} for every $n\geqslant 1$, since slightly modifications (may) have a huge influence on the iterates in the long term, but we are sure that this is true until then $n$-th step, and this is all we need. 

Now we repeat the argument with $\rho_{p_{n+1}}$. Doing this procedure, we  modify the definition of $f$ in every disk of radius 1 centered at each $(x_{k})_{k\geqslant 0}$ infinitely often (one at each step of the recursive process) by the quantity of $t_{n}$. If we choose the sequence $t_{n}$ so that the sum over all $n$'s is less than a prefixed $\varepsilon>0$ we are sure that the limit function under this procedure will satisfy statement \textbf{(d)} in Lemma \ref{lem:construction_Un} for every $n\geqslant 1$, that is     
$$
f^{n+1}(U_{n})\subset U_{n+1},
$$
as desired.
\end{proof}

\begin{proof}[Proof of Theorem \ref{thm:Bishop_example}] It is a straightforward consequence of Lemma \ref{lem:Bishop_example} above.
\end{proof}


\section{Proof of Theorem \ref{thm:A}}\label{sec:proof_thmA}\label{sec:thmA}

Let $f$ be the map obtained from Lemma \ref{lem:Bishop_example}. Recall that every critical value $w_{n}$ lies in $f(\frac{1}{4}\widetilde{D_{n}})$ since $D_{n}=D(z_{n},1)$ and $w_{n}=f(z_{n})$. For every $n\geqslant 1$, assume that either $w_{n}=\frac{1}{2}$ or $f(\frac{1}{4}D_{n})\subset\cup_{m\geqslant 1}U_{m}$. Our goal is to prove  that any wandering  domain $W$ of $f$, if it exists, eventually intersects, under iteration, the $(U_{n})_{n\geqslant 1}$, or their symmetric sets with respect to the real axis (see Lemma \ref{lem:thmA_no_unexpected}). In particular, if the parameters $(\lambda,(d_{n})_{n\geqslant 1},(w_{n})_{n\geqslant 1})$ are chosen as explained in the previous section, the only wandering domains of $f$ are the preimages of the Fatou components which contain the $(U_{n})_{n\geqslant 1}$, and their symmetric ones with respect to the real axis. This will establish Theorem \ref{thm:A}.

We denote by $A^{\star}$ the union of any subset $A$ of the complex plane with its symmetric set with respect to the real axis, equivalently $z\in A^{\star}\iff z\in A\text{ or }\overline{z}\in A$.

The first step of the proof deals with the limit functions of the iterates of $f$ restricted to any wandering domain. 

\begin{lem}\label{lem:thmA_Baker}
Let $W$ be a wandering domain of $f$. Then there exist two subsequences of positive integers $(n_{k})_{k\geqslant 1}$ and $(m_{k})_{k\geqslant 1}$ such that $f^{n_{k}}|_{W}\to\frac{1}{2}$ as $k\to+\infty$ and $f^{n_{k}-1}(W)\subset D_{m_{k}}^{\star}$ for every $k\geqslant 1$.
\end{lem}

\begin{proof}
It is known that all limit functions of $(f^{n}|_{W})_{n\geqslant 1}$ are constant (see \cite[Section 28]{Fatou1}). Since $f$ is in class $\classB$, it follows from \cite[Theorem 1]{DynamicalPropertiesEntireFunctions} that at least one of these constant limit functions is finite. Moreover Baker \cite{LimitFunctions} showed that the finite constant limit functions are in the closure of the postsingular set $E$. This result was improved later in \cite{LimitFunctionsIteratesWanderingDomains} showing that such limit must be in the derived set $E'$ of the postsingular set. In our case $E'=\{f^{\ell}(\frac{1}{2})\}_{\ell\geqslant 0}$, but this is not enough to deduce that $\frac{1}{2}$ is a constant limit function. However, this will come from the fact that $\frac{1}{2}$ is the only critical value in $E'$, following the proof in \cite{LimitFunctions}.

Indeed, let $a=f^{\ell}(\frac{1}{2})$ be a constant limit function of $(f^{n}|_{W})_{n\geqslant 1}$ for some fixed $\ell\geqslant 0$, and let $(n_{k})_{k\geqslant 1}$ be a subsequence of positive integers such that $f^{n_{k}}|_{W}\to a$ as $k\to+\infty$. Let $\neighbor_{a}$ be a small neighborhood of $a$ disjoint from the orbit of $+1$ (and hence $-1$ as well), say $\neighbor_{a}=f^{\ell}(\neighbor_{1/2})$. Assume by contradiction that for every $n_{k}$ large enough, the inverse branch $G_{n_{k}}=(f^{n_{k}}|_{W})^{-1}$, which maps $f^{n_{k}}(W)\subset\neighbor_{a}$ onto $W$, is well defined and may be extended analytically to $\neighbor_{a}$. It follows from \cite[Lemma 1]{LimitFunctions} that the $G_{n_{k}}$ for $n_{k}$ large enough form a normal family. Without loss of generality, we may assume that $G_{n_{k}}\to G$ uniformly on $\neighbor_{a}$ as $k\to+\infty$ where $G$ is holomorphic. Now let $z$ be any point in $W$, and define $z_{n_{k}}\in\neighbor_{a}$ to be $f^{n_{k}}(z)$ for every $n_{k}$ large enough. It follows that $z=G_{n_{k}}(z_{n_{k}})$ tends to the constant $G(a)$ for every $z\in W$, which is a contradiction.

Therefore, up to taking a subsubsequence $(n_{k})_{k\geqslant 1}$ if necessary, we may assume without loss of generality that $G_{n_{k}}$ is not well defined on $\neighbor_{a}$ for every $k\geqslant 1$. Equivalently, there exists $0\leqslant\ell_{k}<n_{k}$ such that at least one critical point of $f$ belongs to the connected component of $(f^{\ell_{k}+1})^{-1}(\neighbor_{a})$ which contains $f^{n_{k}-\ell_{k}-1}(W)$. Such a critical point can not be mapped to $-1$ or $+1$ since $\neighbor_{a}$ is disjoint from the orbit of $+1$ by definition. The critical point is thus of the form $z_{m_{k}}$ or $\overline{z_{m_{k}}}$ for some $m_{k}\geqslant 1$, and the corresponding critical value $w_{m_{k}}$ or $\overline{w_{m_{k}}}$ belongs to the connected component of $(f^{\ell_{k}})^{-1}(\neighbor_{a})$ which contains $f^{n_{k}-\ell_{k}}(W)$. Up to taking a subsubsequence $(n_{k})_{k\geqslant 1}$ if necessary, we may assume without loss of generality that $(m_{k})_{k\geqslant 1}$ is a strictly increasing sequence, and the critical value $w_{m_{k}}$ or $\overline{w_{m_{k}}}$ is in the connected component of $(f^{\ell})^{-1}(\neighbor_{a})$ which contains $\frac{1}{2}$ (because $w_{n}\to\frac{1}{2}$ as $n\to+\infty$ and $a=f^{\ell}(\frac{1}{2})$), namely in $\neighbor_{1/2}$. It follows that $\ell_{k}=\ell$ for every $k\geqslant 1$ and $f^{n_{k}-\ell}(W)\to\frac{1}{2}$ as $k\to+\infty$ by continuity of $f^{\ell}$. Finally, for every $n_{k}$ large enough, $f^{n_{k}-\ell-1}(W)$ lies in the connected component of $f^{-1}(\neighbor_{1/2})$ containing $z_{m_{k}}$ or $\overline{z_{m_{k}}}$, hence in $D_{m_{k}}^{\star}$.
\end{proof}

Notice that the proof above works regardless of how the parameters $(\lambda,(d_{n})_{n\geqslant 1},(w_{n})_{n\geqslant 1})$ are fixed (as long as condition (\ref{eq:parameters_assumption}) in Lemma \ref{lem:construction_Un} holds). It only uses the fact that $w_{n}\to\frac{1}{2}$ as $n\to+\infty$.

Now we can apply the tool introduced by Mihaljevi\'c-Brandt and Rempe-Gillen \cite{AbsenceWanderingDomains} to discard the existence of wandering domains in certain regions of the plane (see discussion in introduction of this paper). The argument splits in two lemmata.

\begin{lem}\label{lem:thmA_Rempe}
For every $n\geqslant 1$, assume  that either $w_{n}=\frac{1}{2}$ or $f(\frac{1}{4}D_{n})\subset\cup_{m\geqslant 1}U_{m}$. Let $(n_{k})_{k\geqslant 1}$ be as in Lemma \ref{lem:thmA_Baker} for a wandering  domain $W$, and let $A$ be the following closed set
$$
A:=\big]-\infty,-\frac{1}{2}\big]\cup\big[\frac{1}{2},+\infty\big[\cup\bigcup_{n\geqslant 1}\bigcup_{k=0}^{n}\overline{f^{k}(U_{n}^{\star})}.
$$
If $f^{n}(W)\cap A=\emptyset$ for every $n\geqslant 0$, then
$$
\dist_{U}\left(f^{n_{k}-1}(W),U\setminus f^{-1}(\D)\right)\to+\infty\text{ as $k\to+\infty$,}$$
where $\dist_{U}$ denotes the hyperbolic distance in $U:=\C\setminus A$.
\end{lem}

\begin{proof}
This is a direct application of \cite[Theorem 4.1]{AbsenceWanderingDomains} (see Theorem \ref{thm:Rempe}). Indeed,  from the definition of $f$ and the $(U_{n})_{n\geqslant 1}$ it follows that  $f(A)\subset A$. Moreover $A$ contains all the critical values of $f$ by assumption. Therefore, if $U$ denotes the hyperbolic domain $\C\setminus A$, we have $U':=f^{-1}(U)\subset U$ and $f|_{f^{-1}(U)}:U'\rightarrow U$ is a holomorphic covering map. From Lemma \ref{lem:thmA_Baker}, $f^{n_{k}}(W)\subset\D$ for every $n_{k}$ large enough. Moreover, $\dist_{U}(f^{n_{k}}(W),U\setminus\D)\to+\infty$ as $k\to+\infty$ because $\dist_{U}(f^{n_{k}}(W),\partial\D)$ is uniformly bounded away from $0$ while $f^{n_{k}}|_{W}\to\frac{1}{2}\in\partial U$ as $k\to+\infty$. The proof is concluded by  applying Theorem \ref{thm:Rempe}.
\end{proof}

The following lemma concludes the proof of Theorem \ref{thm:A}.

\begin{lem}\label{lem:thmA_no_unexpected}
Let $W$ be a wandering domain of $f$ and  assume, for every $n\geqslant 1$, that either $w_{n}=\frac{1}{2}$ or $f(\frac{1}{4}D_{n})\subset\cup_{m\geqslant 1}U_{m}$. Then there exist  integers $n,m\geqslant 1$ such that $f^{n}(W)\cap U_{m}^{\star}\neq\emptyset$.
\end{lem}

\begin{proof}
Assume by contradiction that $f^{n}(W)\cap U_{m}^{\star}=\emptyset$ for every $n,m\geqslant 1$. Since every point of the real axis escapes to infinity under iteration, $f^{n}(W)$ is disjoint from the real axis for every $n\geqslant 0$. It follows that $f^{n}(W)\cap A=\emptyset$ for every $n\geqslant 0$ and Lemma \ref{lem:thmA_Rempe} can be applied. Moreover, from Lemma \ref{lem:thmA_Baker}, $f^{n_{k}-1}(W)$ lies in some $D_{m_{k}}^{\star}$ for every $k\geqslant 1$. Therefore $D_{m_{k}}^{\star}$ must contain some connected component of $A$, namely of the form $f^{j_{k}}(U_{j_{k}}^{\star})\subset\widetilde{D_{j_{k}}^{\star}}$ for some $j_{k}\geqslant 1$. Thus, we have $f^{n_{k}-1}(W)\subset D_{m_{k}}^{\star}=\widetilde{D_{j_{k}}^{\star}}$ for every $k\geqslant 1$. Using Lemma \ref{lem:thmA_Rempe} again, it turns out that $f^{n_{k}-1}(W)$ must intersect $\frac{1}{4}\widetilde{D_{j_{k}}^{\star}}$ for $n_{k}$ large enough. Indeed the Euclidean distance between $f^{n_{k}-1}(W)$ and the boundary of the connected component of $f^{-1}(\D)$ which contains $f^{n_{k}-1}(W)$ remains bounded, therefore Lemma \ref{lem:thmA_Rempe} implies that $f^{n_{k}-1}(W)$ is arbitrarily close to $\partial A\cap \widetilde{D_{j_{k}}^{\star}}$ which is contained in $\frac{1}{4}\widetilde{D_{j_{k}}^{\star}}$ by definition of the disks $(U_{n})_{n\geqslant 1}$ (see Lemma \ref{lem:construction_Un}). Consequently $f^{n_{k}}(W)$ intersects $f(\frac{1}{4}\widetilde{D_{j_{k}}^{\star}})\subset\cup_{m\geqslant 1}U_{m}^{\star}$ by assumption,  which is a contradiction.
\end{proof}


\section{Proof of Theorem \ref{thm:B}}\label{sec:thmB}

The proof is divided in two steps. We first use machinery from Section \ref{sec:Bishop_example} to produce two transcendental entire functions $f$ and $g$ in class $\classB$ such that there is a domain which lies in a periodic Fatou domain for $f$ and $g$ but lies in a wandering domain for $f\circ g$ (compare with \cite[Theorem 4]{CompositionEntireFunctions}). This is the aim of Lemma \ref{lem:thmB_Bishop}. Then we show that $f$ and $g$ have no wandering Fatou domains anywhere in all $\C$, concluding the proof of Theorem \ref{thm:B}. To prove the latter we shall proceed as in Section \ref{sec:thmA}, see Lemma \ref{lem:thmB_no_unexpected}.

\begin{lem}\label{lem:thmB_Bishop}
There exist two transcendental entire functions $f$ and $g$ in class $\classB$ coming from Theorem \ref{thm:prototype} for two choices of the parameters $(\lambda,(d_{n})_{n\geqslant 1},(w_{n})_{\geqslant 1})$ which satisfy  condition {\rm (\ref{eq:parameters_assumption})} in Lemma \ref{lem:construction_Un}, with corresponding sequences of euclidean disks $(U_{n}^{f})_{n\geqslant 1}$ and $(U_{n}^{g})_{n\geqslant 1}$ respectively and same subsequence of positive integers $(p_{n})_{n\geqslant 1}$, such that for every $n\geqslant N$ large enough $U_{n}:=U_{n}^{f}\cap U_{n}^{g}$ is not empty and
$$\left\{\begin{array}{rcl}
	f^{4n+1}(U_{4n\phantom{+1}}) & \subset & U_{4n\phantom{+1}} \\
	f^{4n+2}(U_{4n+1}) & \subset & U_{4n+1} \\
	f^{4n+3}(U_{4n+2}) & \subset & U_{4n+3} \\
	f^{4n+4}(U_{4n+3}) & \subset & U_{4n+4} \\
\end{array}\right.,\quad\left\{\begin{array}{rcl}
	g^{4n+1}(U_{4n\phantom{+1}}) & \subset & U_{4n+1} \\
	g^{4n+2}(U_{4n+1}) & \subset & U_{4n+2} \\
	g^{4n+3}(U_{4n+2}) & \subset & U_{4n+2} \\
	g^{4n+4}(U_{4n+3}) & \subset & U_{4n+3} \\
\end{array}\right.,$$
$$
(f\circ g)^{8n+5}(U_{4n})\subset U_{4n+4},\quad\text{and}\quad\left\{\begin{array}{l} (g\circ f)^{2n}(U_{4n})\subset f^{-1}(U_{4n}) \\ (g\circ f)^{8n+5}\left(f^{-1}(U_{4n})\right)\subset f^{-1}(U_{4n+4}) \end{array}\right..
$$
In particular, the open sets $(U_{n})_{n\geqslant 4N}$ lie in periodic Fatou domains for $f$ and $g$, but lie in wandering Fatou domains for $f\circ g$ and $g\circ f$.
\end{lem}

\begin{proof}
Let $\left(\lambda^{0},(d_{n})_{n\geqslant 1},(w_{n}^{0}:=1/2)_{n\geqslant 1}\right)$ be a choice of parameters given by Lemma \ref{lem:dn}. We will adjust the parameters recursively to get two choices of parameters $\left(\lambda^{0},(d_{n})_{n\geqslant 1},(w_{n}^{f})_{n\geqslant 1}\right)$ and $\left(\lambda^{0},(d_{n})_{n\geqslant 1},(w_{n}^{g})_{n\geqslant 1}\right)$ corresponding to $f$ and $g$ respectively. We can exactly proceed as in the proof of Lemma \ref{lem:Bishop_example}, modifying every $\rho_{p_{n}}$ so that it sends $0$ to the center of the euclidean disk which corresponds to the combinatorics described above. Of course, we can chose larger $d_{p_{n}}$ for both $f$ and $g$ if necessary. Taking care that the sum over all corrections is smaller that a prefixed $\varepsilon>0$, we obtain two transcendental entire functions such that for every $n\geqslant N$ large enough $U_{n}:=U_{n}^{f}\cap U_{n}^{g}$ is not empty and
$$\left\{\begin{array}{rcl}
	f^{4n+1}(U^{f}_{4n\phantom{+1}}) & \subset & U^{f}_{4n\phantom{+1}} \\
	f^{4n+2}(U^{f}_{4n+1}) & \subset & U^{f}_{4n+1} \\
	f^{4n+3}(U^{f}_{4n+2}) & \subset & U^{f}_{4n+3} \\
	f^{4n+4}(U^{f}_{4n+3}) & \subset & U^{f}_{4n+4} \\
\end{array}\right.,\quad\left\{\begin{array}{rcl}
	g^{4n+1}(U^{g}_{4n\phantom{+1}}) & \subset & U^{g}_{4n+1} \\
	g^{4n+2}(U^{g}_{4n+1}) & \subset & U^{g}_{4n+2} \\
	g^{4n+3}(U^{g}_{4n+2}) & \subset & U^{g}_{4n+2} \\
	g^{4n+4}(U^{g}_{4n+3}) & \subset & U^{g}_{4n+3} \\
\end{array}\right..$$
The proof of the last statment is shown in  Figure \ref{fig:thmB}.

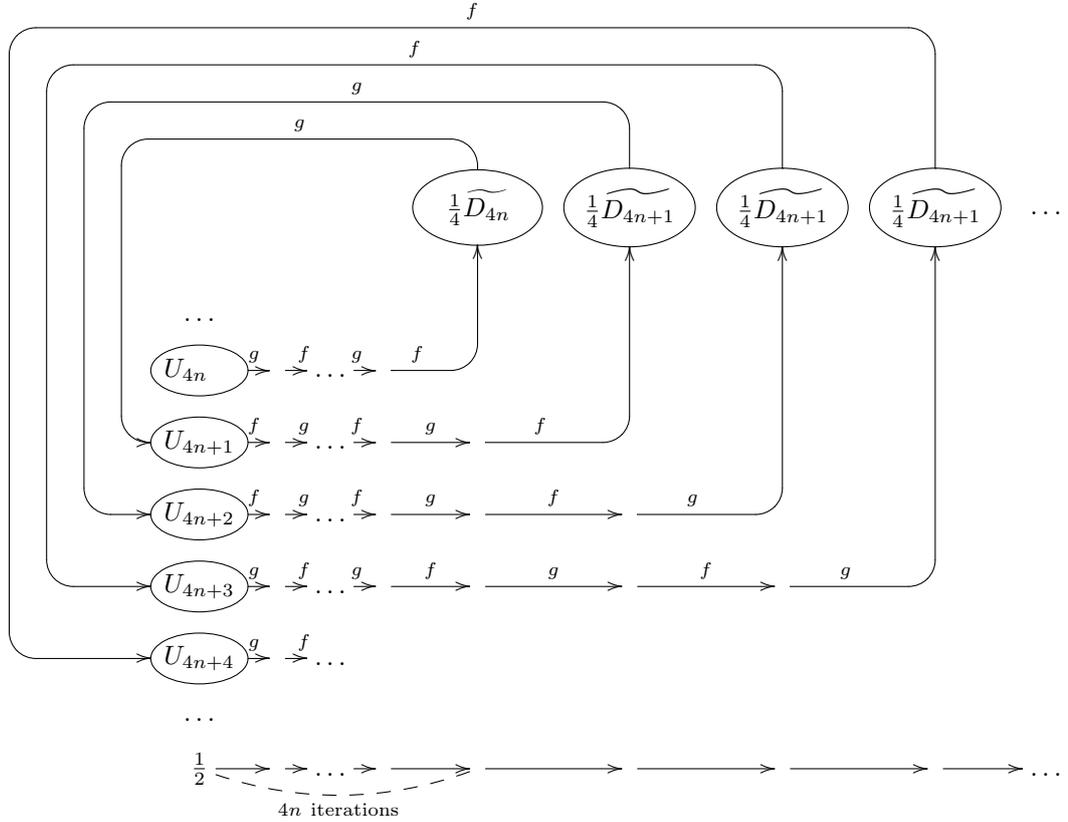
\begin{figure}[!ht]
	\begin{center}
$$\xymatrix @R=8pt @C=8pt {
	&&&&&&&&&&&& \\
	&&&&&&&&&&&& \\
	&&&&&&&&&&&& \\
	&&&&&&&&&&&& \\
	&&&&&&&& *+<15pt>[o][F]{\phantom{1}\frac{1}{4}\widetilde{D_{4n}}_{\phantom{1}}} \ar `u[ulllll] `[ddddlllll]_{g} `[ddddllll] [ddddllll] & *+<15pt>[o][F]{\frac{1}{4}\widetilde{D_{4n+1}}} \ar `u[uulllllll] `[dddddlllllll]_{g} `[dddddlllll] [dddddlllll] & *+<15pt>[o][F]{\frac{1}{4}\widetilde{D_{4n+1}}} \ar `u[uuulllllllll] `[ddddddlllllllll]_{f} `[ddddddllllll] [ddddddllllll] & *+<15pt>[o][F]{\frac{1}{4}\widetilde{D_{4n+1}}} \ar `u[uuuulllllllllll] `[dddddddlllllllllll]_{f} `[dddddddlllllll] [dddddddlllllll] & \dots \\
	&&&&&&&&&&&& \\
	&&&& \dots &&&&&&&& \\
	&&&& *+<10pt>[o][F]{U_{4n\phantom{+1}}} \ar[r]^(0.7){g} & \ar[r]^{f} & \dots \ar[r]^{g} & \ar `r[uuur]^{f} [uuur] &&&&& \\
	&&&& *+<10pt>[o][F]{U_{4n+1}} \ar[r]^(0.7){f} & \ar[r]^{g} & \dots \ar[r]^{f} & \ar[r]^{g} & \ar `r[uuuur]^{f} [uuuur] &&&& \\
	&&&& *+<10pt>[o][F]{U_{4n+2}} \ar[r]^(0.7){f} & \ar[r]^{g} & \dots \ar[r]^{f} & \ar[r]^{g} & \ar[r]^{f} & \ar `r[uuuuur]^{g} [uuuuur] &&& \\
	&&&& *+<10pt>[o][F]{U_{4n+3}} \ar[r]^(0.7){g} & \ar[r]^{f} & \dots \ar[r]^{g} & \ar[r]^{f} & \ar[r]^{g} & \ar[r]^{f} & \ar `r[uuuuuur]^{g} [uuuuuur] && \\
	&&&& *+<10pt>[o][F]{U_{4n+4}} \ar[r]^(0.7){g} & \ar[r]^{f} &\dots &&&&&& \\
	&&&& \dots &&&&&&&& \\
	&&&& \frac{1}{2} \ar[r] \ar@/_{10pt}/@{--}[rrrr]_{4n\ \text{iterations}} & \ar[r] & \dots \ar[r] & \ar[r] & \ar[r] & \ar[r] & \ar[r] & \ar[r] & \dots \\
}$$
	\caption{Sketch (out of scale) of the action of $f$ and $g$.}\label{fig:thmB}
	\end{center}
\end{figure}
\end{proof}

In Section \ref{sec:thmA}, we proved that Bishop's example (coming from Theorem \ref{thm:prototype} like $f$ and $g$ above) has no unexpected wandering domains (Theorem \ref{thm:A}). Using similar arguments we prove the following.

\begin{lem}\label{lem:thmB_no_unexpected}
The two transcendental entire functions $f$ and $g$ coming from Lemma \ref{lem:thmB_Bishop} have no wandering Fatou domains.
\end{lem}

\begin{proof}
If we denote by $A^{f}$ (respectively $A^{g}$) the closed set defined in Lemma \ref{lem:thmA_Rempe} for $f$ (respectively for $g$), we still have $f(A^{f})\subset A^{f}$ (respectively $g(A^{g})\subset A^{g}$) because of Lemma \ref{lem:thmB_Bishop}. This ensures that an analogue of Lemma \ref{lem:thmA_Rempe} holds and the proof is complete. 
\end{proof}


\bibliographystyle{amsalpha}
\bibliography{Biblio}
\addcontentsline{toc}{section}{References}

\end{document}

%% file: FigDomains2.pdf_tex
\begingroup%
  \makeatletter%
  \providecommand\color[2][]{%
    \errmessage{(Inkscape) Color is used for the text in Inkscape, but the package 'color.sty' is not loaded}%
    \renewcommand\color[2][]{}%
  }%
  \providecommand\transparent[1]{%
    \errmessage{(Inkscape) Transparency is used (non-zero) for the text in Inkscape, but the package 'transparent.sty' is not loaded}%
    \renewcommand\transparent[1]{}%
  }%
  \providecommand\rotatebox[2]{#2}%
  \ifx\svgwidth\undefined%
    \setlength{\unitlength}{1659.30850355bp}%
    \ifx\svgscale\undefined%
      \relax%
    \else%
      \setlength{\unitlength}{\unitlength * \real{\svgscale}}%
    \fi%
  \else%
    \setlength{\unitlength}{\svgwidth}%
  \fi%
  \global\let\svgwidth\undefined%
  \global\let\svgscale\undefined%
  \makeatother%
  \begin{picture}(1,0.3591201)%
    \put(0,0){\includegraphics[width=\unitlength]{FigDomains2.pdf}}%
    \put(0,0.30509317){\color[rgb]{0,0,0}\makebox(0,0)[lb]{\smash{\scs$i(\pi+1)$}}}%
    \put(0.02643665,0.25449483){\color[rgb]{0,0,0}\makebox(0,0)[lb]{\smash{$i\pi$}}}%
    \put(0.0289631,0.18807999){\color[rgb]{0,0,0}\makebox(0,0)[lb]{\smash{$i\frac{\pi}{2}$}}}%
    \put(0.03065187,0.10996131){\color[rgb]{0,0,0}\makebox(0,0)[lb]{\smash{$0$}}}%
    \put(0.11884298,0.14671303){\color[rgb]{0,0,0}\makebox(0,0)[lb]{\smash{$S^+$}}}%
    \put(0.17069127,0.28009308){\color[rgb]{0,0,0}\makebox(0,0)[lb]{\smash{$D_1$}}}%
    \put(0.20886921,0.24519166){\color[rgb]{0,0,0}\makebox(0,0)[lb]{\smash{\scs$z_1$}}}%
    \put(0.31570682,0.28021489){\color[rgb]{0,0,0}\makebox(0,0)[lb]{\smash{$D_2$}}}%
    \put(0.3538312,0.24606841){\color[rgb]{0,0,0}\makebox(0,0)[lb]{\smash{\scs$z_2$}}}%
    \put(0.46989935,0.28075046){\color[rgb]{0,0,0}\makebox(0,0)[lb]{\smash{$D_3$}}}%
    \put(0.47871091,0.24755336){\color[rgb]{0,0,0}\makebox(0,0)[lb]{\smash{\scs$z_3$}}}%
    \put(0.18556548,0.09122789){\color[rgb]{0,0,0}\makebox(0,0)[lb]{\smash{$a_1$}}}%
    \put(0.21734137,0.09033694){\color[rgb]{0,0,0}\makebox(0,0)[lb]{\smash{$\pi$}}}%
    \put(0.32976265,0.09059495){\color[rgb]{0,0,0}\makebox(0,0)[lb]{\smash{$a_2$}}}%
    \put(0.35887557,0.09054139){\color[rgb]{0,0,0}\makebox(0,0)[lb]{\smash{$2\pi$}}}%
    \put(0.47816649,0.09081898){\color[rgb]{0,0,0}\makebox(0,0)[lb]{\smash{$a_3$}}}%
    \put(0.50820427,0.09082877){\color[rgb]{0,0,0}\makebox(0,0)[lb]{\smash{$3\pi$}}}%
    \put(0.60841715,0.0600794){\color[rgb]{0,0,0}\makebox(0,0)[lb]{\smash{$\cosh(\lambda\sinh(z))$}}}%
    \put(0.63022769,0.26864243){\color[rgb]{0,0,0}\makebox(0,0)[lb]{\smash{$\rho_n\left((z-z_n)^{d_n}\right)$}}}%
    \put(0.80461147,0.1010358){\color[rgb]{0,0,0}\makebox(0,0)[lb]{\smash{\scs$-1$}}}%
    \put(0.9303393,0.10086364){\color[rgb]{0,0,0}\makebox(0,0)[lb]{\smash{\scs$1$}}}%
    \put(0.86537236,0.09813731){\color[rgb]{0,0,0}\makebox(0,0)[lb]{\smash{\scs$0$}}}%
    \put(0.89592724,0.09503898){\color[rgb]{0,0,0}\makebox(0,0)[lb]{\smash{\scs$\frac12$}}}%
    \put(0.8851381,0.14255039){\color[rgb]{0,0,0}\makebox(0,0)[lb]{\smash{$w_n$}}}%
  \end{picture}%
\endgroup%

%% file: FigGraphT.pdf_tex
\begingroup%
  \makeatletter%
  \providecommand\color[2][]{%
    \errmessage{(Inkscape) Color is used for the text in Inkscape, but the package 'color.sty' is not loaded}%
    \renewcommand\color[2][]{}%
  }%
  \providecommand\transparent[1]{%
    \errmessage{(Inkscape) Transparency is used (non-zero) for the text in Inkscape, but the package 'transparent.sty' is not loaded}%
    \renewcommand\transparent[1]{}%
  }%
  \providecommand\rotatebox[2]{#2}%
  \ifx\svgwidth\undefined%
    \setlength{\unitlength}{1446bp}%
    \ifx\svgscale\undefined%
      \relax%
    \else%
      \setlength{\unitlength}{\unitlength * \real{\svgscale}}%
    \fi%
  \else%
    \setlength{\unitlength}{\svgwidth}%
  \fi%
  \global\let\svgwidth\undefined%
  \global\let\svgscale\undefined%
  \makeatother%
  \begin{picture}(1,0.57123098)%
    \put(0,0){\includegraphics[width=\unitlength]{FigGraphT.pdf}}%
    \put(0.31374777,0.52496357){\color[rgb]{0,0,0}\makebox(0,0)[lb]{\smash{$S_0$}}}%
    \put(0.48087484,0.52393826){\color[rgb]{0,0,0}\makebox(0,0)[lb]{\smash{$S_1$}}}%
    \put(0.64595128,0.52393825){\color[rgb]{0,0,0}\makebox(0,0)[lb]{\smash{$S_2$}}}%
    \put(0.81205305,0.52496359){\color[rgb]{0,0,0}\makebox(0,0)[lb]{\smash{$S_3$}}}%
    \put(0.39269738,0.40295057){\color[rgb]{0,0,0}\makebox(0,0)[lb]{\smash{$D_1$}}}%
    \put(0.55674848,0.40192523){\color[rgb]{0,0,0}\makebox(0,0)[lb]{\smash{$D_2$}}}%
    \put(0.72285025,0.40295059){\color[rgb]{0,0,0}\makebox(0,0)[lb]{\smash{$D_3$}}}%
    \put(0.88895201,0.40295059){\color[rgb]{0,0,0}\makebox(0,0)[lb]{\smash{$D_4$}}}%
    \put(0.46344443,0.23992479){\color[rgb]{0,0,0}\makebox(0,0)[lb]{\smash{$S^+$}}}%
  \end{picture}%
\endgroup%

%% file: FagGodJar_2014.bbl
\newcommand{\etalchar}[1]{$^{#1}$}
\providecommand{\bysame}{\leavevmode\hbox to3em{\hrulefill}\thinspace}
\providecommand{\MR}{\relax\ifhmode\unskip\space\fi MR }
\providecommand{\MRhref}[2]{%
  \href{http://www.ams.org/mathscinet-getitem?mr=#1}{#2}
}
\providecommand{\href}[2]{#2}
\begin{thebibliography}{MBRG13}

\bibitem[Bak63]{MultiplyConnectedDomains}
Irvine~N. Baker, \emph{Multiply connected domains of normality in iteration
  theory}, Mathematische Zeitschrift \textbf{81} (1963), 206--214.

\bibitem[Bak70]{LimitFunctions}
\bysame, \emph{Limit functions and sets of non-normality in iteration theory},
  Annales Academi{\ae} Scientiarum Fennic{\ae} Mathematica \textbf{467} (1970),
  2--11.

\bibitem[Bak76]{EntireFunctionWanderingDomain}
\bysame, \emph{An entire function wich has wandering domains}, Journal of the
  Australian Mathematical Society \textbf{22} (1976), no.~2, 173--176.

\bibitem[Bak84]{WanderingDomains}
\bysame, \emph{Wandering domains in the iteration of entire functions},
  Proceedings of the London Mathematical Society \textbf{49} (1984), no.~3,
  563--576.

\bibitem[Ber93]{IterationMeromorphicFunctions}
Walter Bergweiler, \emph{Iteration of meromorphic functions}, Bulletin of the
  American Mathematical Society \textbf{29} (1993), no.~2, 151--188.

\bibitem[BF13]{QuasiconformalSurgeryBook}
Bodil Branner and N{\'u}ria Fagella, \emph{Quasiconformal surgery in
  holomorphic dynamics}, Cambridge Studies in Advances Mathematics, vol. 141,
  Cambridge University Press, Cambridge, 2013.

\bibitem[BH99]{SemiconjugationEntireFunctions}
Walter Bergweiler and Aimo Hinkkanen, \emph{On semiconjugation of entire
  functions}, Mathematical Proceedings of the Cambridge Philosophical Society
  \textbf{126} (1999), no.~3, 565--574.

\bibitem[BHK{\etalchar{+}}93]{LimitFunctionsIteratesWanderingDomains}
Walter Bergweiler, Mako Haruta, Hartje Kriete, Hans-G{\"u}nter Meier, and
  Norbert Terglane, \emph{On the limit functions of iterates in wandering
  domains}, Annales Academi{\ae} Scientiarum Fennic{\ae} Mathematica
  \textbf{18} (1993), 369--375.

\bibitem[Bis]{QuasiconformalFoldings}
Christopher~J. Bishop, \emph{Constructing entire functions by quasiconformal
  foldings}, to appear in Acta Mathematica.

\bibitem[BRS13]{MultiplyConnectedWanderingDomains}
Walter Bergweiler, Philip~J. Rippon, and Gwyneth Stallard, \emph{Multiply
  connected wandering domains of entire functions}, Proceedings of the London
  Mathematical Society \textbf{107} (2013), no.~6, 1261--1301.

\bibitem[BW98]{CompositeEntireFunctions}
Walter Bergweiler and Wang{ }Yuefei, \emph{On the dynamics of composite entire
  functions}, Arkiv f{\"o}r Matematik \textbf{36} (1998), no.~1, 31--39.

\bibitem[EL87]{EntireFunctionsPathologicalDynamics}
Alexandre~E. Er{\`e}menko and Mikhail Lyubich, \emph{Examples of entire
  functions with pathological dynamics}, Journal of the London Mathematical
  Society \textbf{36} (1987), no.~3, 458--468.

\bibitem[EL92]{DynamicalPropertiesEntireFunctions}
\bysame, \emph{Dynamical properties of some classes of entire functions},
  Annales de l'Institut Fourier \textbf{42} (1992), no.~4, 989--1020.

\bibitem[Er{\`e}89]{IterationEntireFunctions}
Alexandre~E. Er{\`e}menko, \emph{On the iteration of entire functions},
  Dynamical systems and ergodic theory ({W}arsaw, 1986), Banach Center
  Publications, vol.~23, PWN, 1989, pp.~339--345.

\bibitem[Fat19]{Fatou1}
Pierre Fatou, \emph{Sur les {\'e}quations fonctionnelles}, Bulletin de la
  Soci{\'e}t{\'e} Math{\'e}matique de France \textbf{47} (1919), 161--271.

\bibitem[GK86]{FinitenessTheoremEntireFunctions}
Lisa~R. Goldberg and Linda Keen, \emph{A finiteness theorem for a dynamical
  class of entire functions}, Ergodic Theory and Dynamical Systems \textbf{6}
  (1986), no.~2, 183--192.

\bibitem[Her84]{ExemplesRationnellesOrbiteDense}
Micha{\"e}l Herman, \emph{Exemples de fractions rationnelles ayant une orbite
  dense sur la sph{\`e} de {R}iemann}, Bulletin de la Soci{\'e}t{\'e}
  Math{\'e}matique de France \textbf{112} (1984), no.~1, 93--142.

\bibitem[KS08]{OnMultiplyConnectedWanderingDomains}
Masashi Kisaka and Mitsuhiro Shishikura, \emph{On multiply connected wandering
  domains of entire functions}, Transcendental dynamics and complex analysis,
  London Mathematical Society Lecture Note Series, vol. 348, Cambridge
  University Press, 2008, pp.~217--250.

\bibitem[MBRG13]{AbsenceWanderingDomains}
Helena Mihaljevi{{\'c}}-Brandt and Lasse Rempe-Gillen, \emph{Absence of
  wandering domains for some real entire functions with bounded singular sets},
  Math. Ann. \textbf{357} (2013), no.~4, 1577--1604. \MR{3124942}

\bibitem[Mil06]{MilnorBook}
John Milnor, \emph{Dynamics in one complex variable}, third ed., Annals of
  Mathematics Studies, vol. 160, Princeton University Press, Princeton, NJ,
  2006.

\bibitem[Pom75]{PommerenkeBook}
Christian Pommerenke, \emph{Univalent functions}, Studia
  Mathematica/Mathematische Lehrb{\"u}cher, vol.~15, Vandenhoeck und Ruprecht,
  G{\"o}ttingen, 1975.

\bibitem[Sin03]{CompositionEntireFunctions}
Anand~Prakash Singh, \emph{On the dynamics of composition of entire functions},
  Mathematical Proceedings of the Cambridge Philosophical Society \textbf{134}
  (2003), no.~1, 129--138.

\bibitem[Sul85]{NoWanderingTheorem}
Denis Sullivan, \emph{Quasiconformal homeomorphisms and dynamics. {P}art {I}.
  {S}olution of the {F}atou-{J}ulia problem on wandering domains}, Annals of
  Mathematics \textbf{122} (1985), no.~3, 401--418.

\bibitem[War42]{ConformalMappingInfiniteStrips}
Stefan~E. Warschawski, \emph{On conformal mapping of infinite strips},
  Transactions of the American Mathematical Society \textbf{51} (1942),
  280--335.

\end{thebibliography}
